\documentclass{article}
\usepackage{amssymb}

\usepackage{makeidx}
\usepackage{graphicx}
\usepackage{amsmath}
\usepackage[english]{babel}

\newtheorem{theorem}{Theorem}

\newtheorem{corollary}[theorem]{Corollary}

\newtheorem{lemma}[theorem]{Lemma}

\newenvironment{proof}[1][Proof]{\textbf{#1.} }{\ \rule{0.5em}{0.5em}}

\begin{document}

\title{Partial Probability and Kleene Logic}
\author{Maurizio Negri, \\ Universit\`{a} di Torino}
\maketitle
\begin{abstract}
There are two main approaches to probability, one of set-theoretic character 
where probability is the measure of a set, and another one of linguistic character 
where probability is the degree of confidence in a proposition.
In this work we give an unified algebraic treatment of these
approaches through the concept of valued lattice, obtaining as a by-product a 
translation between them. Then we introduce the concept of
partial valuation for DMF-algebras (De Morgan algebras with a single fixed point for negation),
giving an algebraic setting for probability of partial events.
We introduce the concept of partial probability for propositions,
substituting classical logic with Kleene's logic. In this case too we give a 
translation between set-theoretic and linguistic probability. 
Finally, we introduce the concept of conditional partial probability and prove a weak
form of Bayes's Theorem.

\vspace{5mm} \noindent \textbf{Keywords:} non-classical probability; 
Kleene's Logic; De Morgan Algebras; Valued Lattices; Bayes's Theorem.
\end{abstract}

\section{Probability and logic}

People are generally introduced to probability through the concept of a
probability space, a triple \ $(A,\mathcal{C}_{A},p)$ where $A$ is a sample
space, $\mathcal{C}_{A}$ a field of sets over $A$ and $p:\mathcal{C}%
_{A}\rightarrow \lbrack 0,1]$ a function satisfying Kolmogoroff's axioms.
This is the set-theoretic approach, where events are classical sets
belonging to a field of sets and probability is the measure of a set. \
There is, however, another approach of linguistic character, where the
bearers of probability are sentences and the probability, or degree of
confidence, is measured by a probability function. We say that a function $%
\pi $ from the set $F$ of formulas of a sentential language to $[0,1]$ is a
probability function, if the following axioms are satisfied:

\begin{enumerate}
\item  if $\models \alpha $ then $\pi (\alpha )=1$,

\item  if $\models \lnot (\alpha \wedge \beta )$ then $\pi (\alpha \vee
\beta )=\pi (\alpha )+\pi (\beta )$,
\end{enumerate}

\noindent where $\models $ is the consequence relation of bivalent logic.
(For the linguistic approach to classical probability, see, for instance, 
\cite[ch. 2, par. c.1]{howson1989} 
or par. \ref{parprobsent} below). Both
point of view on classical probability are deeply rooted in Boolean
algebras, on one side for the notion of set of classical set theory, on the
other side for the notion of logical consequence of bivalent logic.

If the classic approach to probability is modified by the introduction of
partial events, as in \cite{negri2010}, the algebra of events becomes a
DMF-algebra, a De Morgan algebra with a single fixed point for negation. We
define partial probability spaces in par. \ref{parprobpar} by a
set of axioms that the partial probability measure must obey. The linguistic
approach to partial probability is introduced in par. \ref{parprobparsent},
where a partial probability function is given by means of axioms in which
the consequence relation $\models$ is borrowed from Kleene logic.

The aim of this work is to give an unified algebraic treatment of these
subjects through the concept of valued lattice (see \cite[ch. 10]{birkhoff1967}). 
In par. 2 we give the algebraic counterpart of classical
probability and in par. 5-8 we introduce the algebraic tools for the 
development of partial
probability. In this way, the passage from classical to partial probability
can be seen as the shift from Boolean algebras to DMF-algebras. As a result,
we obtain a translation from the set-theoretic treatment of probability in
term of events to the linguistic one in terms of sentences and vice versa, \
both in the case governed by Boolean algebras (bivalent logic) and in the
case governed by DMF-algebras (Kleene logic). This is the subject of par. 4,
11 and 12. In par. 13 we give a weak form of Bayes's Theorem and a result
about conditional partial probability.

In the following we shall be only concerned with a finitary notion of
probability, so we confine ourselves to probability spaces with a finite
sample space and to sentential languages with a finite number of sentential
variables.

\section{Valuations\label{parboolval}}

A general setting for probability spaces can be given
through the notion of valuation and valued lattice. If $\mathcal{A}$ is a
lattice we say that $v:A\rightarrow R$ is a \textit{valuation} on $\mathcal{A%
}$ if 
\begin{equation}
v(a\vee b)=v(a)+v(b)-v(a\wedge b).  \tag{$\ast$}  \label{1}
\end{equation}
If $x\leq y$ implies $v(x)\leq v(y)$, we say that $v$ is \textit{isotone}; $%
v $ is \textit{strictly isotone} if we can substitute $\leq $ with $<$.
(Birkhoff calls \textit{positive} a strictly isotone valuation.) In the
following we will confine ourselves to non-negative valuations, i.e.
valuations such that $0\leq v(a)$, for all $a\in A$. A \textit{valued lattice%
} is a pair $(\mathcal{A},v )$ where $\mathcal{A}$ is a lattice and $v $
a valuation on $\mathcal{A}$.

Let $L_{0,1}$ be the class of bounded lattices. If $\mathcal{A}\in L_{0,1}$,
we say that $v $ is a \textit{bounded lattices valuation} if $v $ is a
valuation on $\mathcal{A}$ and $v (0)=0$ and $v (1)=1$. If $v $ is an
isotone valuation on a bounded lattice $\mathcal{A}$, then $v \lbrack
A]\subseteq \lbrack 0,1]$. A \textit{valued bounded lattice} is a pair $(%
\mathcal{A},v )$ where $\mathcal{A}$ is a bounded lattice and $v $ a
bounded lattices valuation on $\mathcal{A}$.

Let $BA$ be the class of Boolean algebras. If $\mathcal{A}\in BA$, we say
that $v $ is a \textit{Boolean valuation} if $v $ is a bounded lattices
valuation on $\mathcal{A}$. A \textit{valued Boolean algebra} is a pair $(%
\mathcal{A},v )$, where $\mathcal{A}$ is a Boolean algebra and $v $ a
Boolean valuation on $\mathcal{A}$. We can give an equivalent definition of
a Boolean valuation as follows. In a lattice with $0$, a function $%
f:A\rightarrow \lbrack 0,1]$ is said to be \textit{additive} iff $f(a\vee
b)=f(a)+f(b)$, whenever $a\wedge b=0$.

\begin{theorem}
\label{teoadd}If $\mathcal{A}\in BA$ and $v :A\rightarrow \lbrack 0,1]$,
then $v $ is a Boolean valuation iff $v (1)=1$ and $v $ is additive.
\end{theorem}

\begin{proof}
If $v$ is a Boolean valuation then $v(1)=1$. As for additivity, we assume $%
a\wedge b=0$ then $v (a\vee b)=v (a)+v (b)$ follows from (\ref{1}) and
from $v(0)=0$. In the other direction, we assume that $v$ is additive and $%
v(1)=1$. We show that $v $ is a bounded latticed valuation. As $a\vee 0=a$
and $a\wedge 0=0$, by additvity we have $v(a)=v(a)+v(0)$ and so $v(0)=0$.
Finally we show that (\ref{1}) holds. In every Boolean algebra, $a\vee b=a\vee
(b-a)$ and $a\wedge (b-a)=0$, so by additivity we have $v(a\vee
b)=v(a)+v(b-a)$. In every Boolean algebra, $b=(b-a)\vee (a\wedge b)$ and $%
(b-a)\wedge (a\wedge b)=0$, so by additivity we have $v(b)=v(b-a)+v(a\wedge
b)$. In conclusion, $v(a\vee b)=v(a)+v(b)-v(a\wedge b)$.
\end{proof}

So a probability space $\ (A,\mathcal{C}_{A},p)$, where $A$ is a finite
sample space, $\mathcal{C}_{A}$ a field of sets on $A$ and $p$ a probability
measure satisfying Kolmogoroff's axioms with finite additivity, is a
particular case of valued Boolean algebra. In the following theorems 
some elementary properties of valuations are collected.

\begin{theorem}
\label{teo222} \hfill 

\begin{enumerate}
\item  If $\ \mathcal{A},\mathcal{B}\in L_{0,1}$, $(\mathcal{B},v)$ is a
valued bounded lattice and $\varphi :\mathcal{A}\rightarrow \mathcal{B}$ is
a morphism of bounded lattices, then $(\mathcal{A},v\circ \varphi )$ is a
valued bounded lattice$.$

\item  If $\ \mathcal{A},\mathcal{B}\in BA$, $(\mathcal{B},v)$ is a valued
Boolean algebra and $\varphi :\mathcal{A}\rightarrow \mathcal{B}$ is a
morphism of bounded lattices, then $(\mathcal{A},v\circ \varphi )$ is a
valued Boolean algebra$.$
\end{enumerate}
\end{theorem}

\begin{proof}
1. As $\varphi $ is a bounded lattices morphism and $v$ is a bounded
lattices valuation, $v(\varphi (1^{\mathcal{A}}))=v(1^{\mathcal{B}})=1$. In
the same way, $v(\varphi (0^{\mathcal{A}}))=0$. We prove that $v\circ
\varphi $ satisfies (\ref{1}): 
\begin{eqnarray*}
v\circ \varphi (x\vee y) &=&v(\varphi (x)\vee \varphi (y)) \\
&=&v(\varphi (x))+v(\varphi (y))-v(\varphi (x)\wedge \varphi (y)) \\
&=&v(\varphi (x))+v(\varphi (y))-v(\varphi (x\wedge y)) \\
&=&v\circ \varphi (x)+v\circ \varphi (y)-v\circ \varphi (x\wedge y).
\end{eqnarray*}
2. We observe that \ a morphism of bounded lattices is also a morphism of
Boolean algebras: if $\wedge ,\vee $, $0$, $1$ are preserved by \ then $%
\lnot $ is preserved too. The result follows by an analogous proof.
\end{proof}

\begin{theorem}
\label{teoboolval}If $(\mathcal{A},v)$ is a Boolean valued algebra then

\begin{enumerate}
\item  $v$ is isotone,

\item  $v(\lnot a)=1-v(a)$.
\end{enumerate}
\end{theorem}

\begin{proof}
1. If $a\leq b$ then $b=a\vee (b\wedge \lnot a)$ and $0=a\wedge (b\wedge \lnot
a)$. So $v(b)=v(a)+v(b\wedge \lnot a)$ and $v(a)\leq v(b)$ follows.

2. As $v(a\wedge \lnot a)=0$ e $v(a\vee \lnot a)=1$, we have $1=$ $v(a\vee
\lnot a)=v(a)+v(\lnot a)$
\end{proof}

As $v$ is isotone, every Boolean valuation takes its values in $[0,1]$.

Let $(\mathcal{A},v)$ and $(\mathcal{B},\mu )$ be valued bounded lattices
(valued Boolean algebras). We say that $\varphi :A\rightarrow B$ is a 
\textit{valued bounded lattices morphism} (\textit{valued Boolean algebra
morphism}) if $\varphi $ is a bounded lattices morphism and, for all $a\in A$%
, $v(a)=\mu (\varphi (a))$, i.e. $\varphi $ preserves not only the algebraic
structure, but also the valuations of individuals. We say that the\ valued
bounded lattices (valued Boolean algebras) $(\mathcal{A},v)$ and $(\mathcal{B%
},\mu )$ are \textit{isomorphic} if there is an isomorphism $\varphi
:A\rightarrow B$ of valued bounded lattices (valued Boolean algebras).

\begin{theorem}
\label{teovalind}If $\varphi :\mathcal{A}\rightarrow \mathcal{B}$ is a
bounded lattices isomorphism and $(\mathcal{A},v)$ is a bounded
lattice then $\mu =v\circ \varphi ^{-1}$ is a valued bounded lattices
valuation on$\ \mathcal{B}$ and $(\mathcal{A},v)$ and $(\mathcal{B},\mu )$
are isomorphic in $\varphi $ as valued bounded lattices.
\end{theorem}

\begin{proof}
By theorem \ref{teo222} we can define a valuation $\mu $ on $\mathcal{B}$
setting $\mu (x)=v(\varphi ^{-1}(x))$. Now it can be easily seen that $(%
\mathcal{A},v)$ and $(\mathcal{B},\mu )$ are isomorphic in $\varphi $ as
valued bounded lattices: by hypothesis $\varphi $ is a bounded
lattices isomorphism and, for all $a\in A$, $v(a)=v(\varphi ^{-1}(\varphi
(a)))=\mu (\varphi (a))$.
\end{proof}

When $\varphi :\mathcal{A}\rightarrow \mathcal{B}$ and $(\mathcal{A},v)$ are
as above, we say that $\mu $ is the \textit{valuation induced} \textit{on} $%
\mathcal{B}$ \textit{by} $(\mathcal{A},v)$ \textit{and} $\varphi $.

In every bounded lattice $\mathcal{A}$ we can associate to every $a\in A$
the surjection $f_{a}:A\rightarrow \lbrack 0,a]$ defined by $%
f_{a}(x)=a\wedge x$. We call $f_{a}$ the \textit{relativization associated}
to $a$. The closed interval $[0,a]$ is the domain of a bounded lattice $%
\mathcal{B}$ with $0^{\mathcal{B}}=0^{\mathcal{A}}$ and $1^{\mathcal{B}}=a$.
(Meet and join are inherited from $\mathcal{A}$.) If $\mathcal{A}$ is
distributive then $f_{a}$ is a bounded lattices epimorphism from $\mathcal{A}
$ to $\mathcal{B}$.

The same result can be obtained when $\mathcal{A}$ is a Boolean algebra.
Firstly we must make $[0,a]$ into a Boolean algebra $\mathcal{B}$ by
defining meet, join, top and bottom as above and setting $\lnot ^{\mathcal{B}%
}(x)=a\wedge \lnot ^{\mathcal{A}}(x)$: the complement of $x$ in $\mathcal{B}$
is just the complement of $x$ in $\mathcal{A}$ relativized to $[0,a]$. It
can be easily proved that $f_{a}$ is a Boolean epimorphism from $\mathcal{A}$
to $\mathcal{B}$.

Now we study the behavior of valuations with respect to relativizations. If $%
v$ is a valuation on $\mathcal{A}$ we define, for every $a\in A$ such that $%
v(a)\neq 0$, a function $v_{a}:[0,a]\rightarrow \lbrack 0,1]$ setting 
\begin{equation*}
v_{a}(x)=v(x)\cdot \frac{1}{v(a)}.
\end{equation*}
We call $v_{a}$ the \textit{relativized valuation} associated to $a$ and $v$.

\begin{theorem}
\label{teorelativizedval} For every valued Boolean algebra $(\mathcal{A},v)$
and for all $a\in A$ such that $v(a)\neq 0$, $v_{a}$ is a valuation on $[0,a]
$ and $([0,a],v_{a})$ is a valued Boolean algebra.
\end{theorem}

\begin{proof}
By definition, we have $v_{a}(0)=0$ and $v_{a}(a)=1$. For all $x$, $y\leq a$,
\begin{eqnarray*}
v_{a}(x\vee y)&=&\frac{v(x\vee y)}{v(a)} \\
&=&\frac{v(x)}{v(a)}+\frac{v(y)}{v(a)}-\frac{v(x\wedge y)}{v(a)} \\
&=&v_{a}(x)+v_{a}(y)-v_{a}(x\wedge y),
\end{eqnarray*}
so $v_{a}$ is a valuation on $[0,a]$.
\end{proof}

We suppose $v(a)\neq 0$. As $f_{a}$ is a morphism from $\mathcal{A}$ to $%
[0,a]$ and $v_{a}$ is a valuation on $[0,a]$, the function $%
v(-|a)=v_{a}\circ f_{a}$ is a valuation on $\mathcal{A}$, by theorem \ref
{teo222}. We call $v(-|a)$ the \textit{conditional valuation} associated to $%
v$ and $a$. It can be immediately seen that the concept of conditional
probability is only a particular case of \ conditional valuation: when $%
\mathcal{A}$ is an algebra of events and $v$ is a probability measure on $%
\mathcal{A}$, $\ v(x|a)=v(x\wedge a)/v(a)$ is the conditional probability of 
$x$ with respect to $a$.

\section{Probability of sentences\label{parprobsent}}

In classical probability theory, events are represented by sets and the
probability value of an event can be understood as the measure of a set.
This set theoretic picture of probability can be replaced by a linguistic
one where sentences are the bearers of probability. From this point of view,
it is natural to conceive the number attached to a sentence $\alpha $ as a
degree of belief, representing the extent to which you believe it likely
that $\alpha $ will turn out to be true.

In the following we will denote with $L_{n}$ a sentential $n$-ary language
based on the sentential variables $P_{n}=\{p_{1},...,p_{n}\}$, the
connectives $\{\lnot ,\wedge ,\vee \}$ and the constants $\{0,1\}$. We
denote with $F_{n}$ the set of formulas of $L_{n}$. (We write simply $L$, $P$
and $F$ when no confusion is possible.) We say that $\pi :F\rightarrow
\lbrack 0,1]$ is a \textit{probability function on} $L$ if the following
axioms are satisfied:

\begin{enumerate}
\item  if $\models \alpha $ then $\pi (\alpha )=1$,

\item  if $\models \lnot (\alpha \wedge \beta )$ then $\pi (\alpha \vee
\beta )=\pi (\alpha )+\pi (\beta )$,
\end{enumerate}

\noindent where Greek letters $\alpha ,\beta ,...$ are metalinguistic
variables for formulas. We say that $\alpha $ and $\beta $ are \textit{%
incompatible} iff $\models \lnot (\alpha \wedge \beta )$. In the following
theorem some fundamental properties of $\pi $ are collected.

\begin{theorem}
\label{teoprob}If $\pi $ is a probability function on $L$, then

\begin{enumerate}
\item  $\pi (\lnot \alpha )=1-\pi (\alpha ),$

\item  $\models \alpha \longleftrightarrow \beta $ implies $\pi (\alpha
)=\pi (\beta )$,

\item  $\models \alpha \rightarrow \beta $ implies $\pi (\alpha )\leq \pi
(\beta )$,

\item  $\pi (\alpha \vee \beta )=\pi (\alpha )+\pi (\beta )-\pi (\alpha
\wedge \beta )$.
\end{enumerate}
\end{theorem}

\begin{proof}
1. We have $\models \lnot (\alpha \wedge \lnot \alpha )$, so $\pi (\alpha
)+\pi (\lnot \alpha )=\pi (\alpha \vee \lnot \alpha )$ by axiom 2. On the
other side $\models \alpha \vee \lnot \alpha $, so $\pi (\alpha )+\pi (\lnot
\alpha )=1$ by axiom 1. Then $\pi (\lnot \alpha )=1-\pi (\alpha ).$

2. From our hypothesis we have $\models \lnot \alpha \vee \beta $, so $1=\pi
(\lnot a\vee \beta )$ by axiom 1. From our hypothesis we have $\models \alpha \vee
\lnot \beta $ and then $\models \lnot (\lnot \alpha \wedge \beta )$.
Thus, by axiom 2, we have $\pi (\lnot \alpha \vee \beta )=\pi (\lnot \alpha
)+\pi (\beta )$ and 
\begin{equation*}
1=\pi (\lnot a\vee \beta )=\pi (\lnot \alpha )+\pi (\beta )=1-\pi (\alpha
)+\pi (\beta ),
\end{equation*}
so $\pi (\alpha )=\pi (\beta )$.

3. The following is an easy proposition of classical logic: if $\models
\alpha \rightarrow \beta $ then there is a $\gamma $ such that: 
\begin{eqnarray*}
\text{ i)}\models \alpha \vee \gamma \leftrightarrow \beta \text{,} \\
\text{ii)}\models \lnot (\alpha \wedge \gamma )\text{.}
\end{eqnarray*}
(Set $\gamma =\beta \wedge \lnot \alpha $.) \ So from i) and point 2) above, we
have $\pi (\alpha \vee \gamma )=\pi (\beta )$. From ii) and from axiom 2, we
have $\pi (\alpha \vee \gamma )=\pi (\alpha )+\pi (\gamma )$. So $\pi (\beta
)=\pi (\alpha )+\pi (\gamma )$ and $\pi (\alpha )\leq \pi (\beta )$.

4. We can prove the following equations:

as $\models (\alpha \vee \beta
)\leftrightarrow (\alpha \vee (\beta \wedge \lnot \alpha ))$, by point 2) above 
we have
\begin{equation*}
\pi (\alpha \vee \beta )=\pi (\alpha \vee (\beta \wedge \lnot \alpha )),
\end{equation*}
as $\models \lnot (\alpha \wedge (\beta \wedge \lnot \alpha ))$, by axiom 2
we have 
\begin{equation*}
\pi (\alpha \vee (\beta \wedge \lnot \alpha ))=\pi (\alpha )+\pi (\beta
\wedge \lnot \alpha ),
\end{equation*}
as $\models \lnot ((\beta \wedge \lnot \alpha )\wedge (\alpha \wedge \beta ))
$, by axiom 2 we have 
\begin{equation*}
\pi ((\beta \wedge \lnot \alpha )\vee (\alpha \wedge \beta ))=\pi (\beta
\wedge \lnot \alpha )+\pi (\alpha \wedge \beta ),
\end{equation*}
as $\models ((\beta \wedge \lnot \alpha )\vee (\alpha \wedge \beta
))\leftrightarrow \beta $, by point 2) above we have 
\begin{equation*}
\pi ((\beta \wedge \lnot \alpha )\vee (\alpha \wedge \beta ))=\pi (\beta ).
\end{equation*}
So we have $\pi (\alpha \vee \beta )=\pi (\alpha )+\pi (\beta \wedge \lnot
\alpha )$, by the first two equations, and $\pi (\beta )=\pi (\beta \wedge
\lnot \alpha )+\pi (\alpha \wedge \beta )$ by the last two. Then we can
conclude with $\pi (\alpha \vee \beta )=\pi (\alpha )+\pi (\beta )-\pi
(\alpha \wedge \beta )$.$\medskip $
\end{proof}

For every probability function $\pi $ on $L$ and for all $\delta $ such that 
$\pi (\delta )\neq 0$, we define the \textit{conditional probability with
respect to} $\delta $ as a 1-ary function $\pi (-|\delta ):F\rightarrow
\lbrack 0,1]$ setting 
\begin{equation*}
\pi (\alpha |\delta )=\frac{\pi (\alpha \wedge \delta )}{\pi (\delta )}.
\end{equation*}
As $\delta $ varies over sentences that satisfy $\pi (\delta )\neq 0$, we
can see $\pi (x|y)$ as a 2-ary function. The restriction on the second
argument cannot be avoided, unless we are disposed to accept conditional
probability as a partial function.

Before proving a result similar to theorem \ref{teoprob}, we introduce a
generalization of the concept of probability function. If we observe axioms
1) and 2) in the definition of probability function, it is clear the
fundamental role of logical truth. If we substitute `logical truth' with
`consequence of a set of formulas $\Gamma $' we arrive at a relativized
concept of probability function, characterized by the two following facts:
i) not only every tautology, but also every logical consequence of $\Gamma $
has probability $1$, ii) two formulas may be considered incompatible not
only with respect to logic (absolutely incompatible), but also with respect
\ to a set of formulas $\Gamma $. Then we can define, for every $\delta $
such that $\delta \nvDash 0$, the concept of \textit{probability function on}
$L$ \textit{relative to} $\delta $ as a function $\pi _{\delta
}:F\rightarrow \lbrack 0,1]$ satisfying the following axioms:

\begin{enumerate}
\item  if $\delta \models \alpha $ then $\pi _{\delta }(\alpha )=1$,

\item  if $\delta \models \lnot (\alpha \wedge \beta )$ then $\pi _{\delta
}(\alpha \vee \beta )=\pi _{\delta }(\alpha )+\pi _{\delta }(\beta )$.
\end{enumerate}

\noindent If $\models \delta $ then $\pi _{\delta }$ is simply a probability
function. If $\pi _{\delta }$ is a probability function relative to $\delta $%
, then $\pi _{\delta }$ is a probability function, because $\models \alpha $
implies $\delta \models \alpha $. We could have defined the still more
general notion of a \ probability function `relative to a set of sentences $%
\Gamma $', but there is no point in doing so in the context $n$-ary
languages. It can be easily seen that, if $L$ contains only a finite number
of variables, for every set of sentences $\Gamma $ there is a formula $%
\delta $ logically equivalent to $\Gamma $.

The following theorem \ is analogous to theorem \ref{teoprob} and the proof
is similar.

\begin{theorem}
\label{teoprobcond}If $\pi _{\delta }$ is a probability function relative tu 
$\delta $, then

\begin{enumerate}
\item  $\pi _{\delta }(\lnot \alpha )=1-\pi _{\delta }(\alpha ),$

\item  $\delta \models \alpha \longleftrightarrow \beta $ implies $\pi
_{\delta }(\alpha )=\pi _{\delta }(\beta )$,

\item  $\delta \models \alpha \rightarrow \beta $ implies $\pi _{\delta
}(\alpha )\leq \pi _{\delta }(\beta )$,

\item  $\pi _{\delta }(\alpha \vee \beta )=\pi _{\delta }(\alpha )+\pi
_{\delta }(\beta )-\pi _{\delta }(\alpha \wedge \beta )$.
\end{enumerate}
\end{theorem}

\begin{theorem}
The conditional probability $\pi (-|\delta )$ is a probability function
relative to $\delta $ and then a probability function.
\end{theorem}

\begin{proof}
In order to show that $\pi (x|\delta )$ is a probability function relative
to $\delta $, we start proving that $\delta \nvDash 0$. Suppose toward a
contradiction that $\delta \models 0$, then $1\models \lnot \delta $ and $%
\pi (\lnot \delta )=1$ and so $\pi (\delta )=0$. This is absurd, because the
conditional probability $\pi (x|\delta )$ requires $\pi (\delta )\neq 0$. \
Then we verify that $\pi (x|\delta )$ takes values in $[0,1]$. On one side,
for all $\alpha $ we have $0\leq \pi (\alpha |\delta )$, as $\pi (\alpha
\wedge \delta )$, $\pi (\delta )\geq 0$. On the other side, $\alpha \wedge
\delta \models \delta $ implies $\pi (\alpha \wedge \delta )\leq \pi (\delta
)$, by point 3) of theorem \ref{teoprob}, so $\pi (\alpha |\delta )\leq 1$.
We verify the first axiom: if $\delta \models \alpha $ then $\models \delta
\longleftrightarrow \alpha \wedge \delta $ and, by point 2) of theorem \ref
{teoprob}, $\pi (\alpha \wedge \delta )=\pi (\delta )$ and so $\pi (\alpha
|\delta )=1$. We verify the second axiom. We suppose that $\delta \models
\lnot (\alpha \wedge \beta )$. Then $\ $%
\begin{eqnarray*}
\pi (\alpha \vee \beta |\delta )&=&\frac{\pi ((\alpha \vee \beta )\wedge
\delta )}{\pi (\delta )} \\
&=&\frac{\pi ((\alpha \wedge \delta )\vee (\beta \wedge \delta ))}{\pi (\delta
)} \\
&=&\frac{\pi (\alpha \wedge \delta )}{\pi (\delta )}+\frac{\pi (\beta \wedge
\delta )}{\pi (\delta )} \\
&=&\pi (\alpha |\delta )+\pi (\beta |\delta ).
\end{eqnarray*}
The second line follows from point 2) of theorem \ref{teoprob}. The third line
follows from the second axiom on $\pi $ because $\models \lnot ((\alpha
\wedge \delta )\wedge (\beta \wedge \delta ))$ holds. In fact, from our
hypothesis $\models \delta \rightarrow \lnot (\alpha \wedge \beta )$ holds
and 
\begin{equation*}
\models (\delta \rightarrow (\lnot \alpha \vee \lnot \beta
))\longleftrightarrow ((\lnot \delta \vee \lnot \alpha )\vee (\lnot \delta
\vee \lnot \beta ))\longleftrightarrow (\lnot (\delta \wedge \alpha )\vee
\lnot (\delta \wedge \beta )).
\end{equation*}
\end{proof}

\section{Translatability\label{partrad}}

There are two fundamental ways of understanding probability, as the measure
of a set representing an event and as the degree of belief in a sentence
describing an event: we show that these two ways of
understanding probability can be translated one into the other. As we are
concerned with two kinds of subjects bearers of probability, respectively
sets and sentences, we'll find a common ground in the realm of Boolean
algebras, \ where the set-theoretic aspect of events is naturally \
represented by fields of sets and the logical-linguistic aspect of events is
represented by Lindenbaum algebras.

Suppose that the notion of probability be given as probability of sentences, by
means of a probability function $\pi $ on $L_{n}$. We aim to define a
probability space $(A,\mathcal{P}(A),p)$ where the probability given by $\pi 
$ is translated in terms of measure of sets; namely we aim to define a
function $\varphi $ that takes every formula $\alpha \in L_{n}$ to an event $%
\varphi (\alpha )$ of $\mathcal{P}(A)$ in such a way that $\pi (\alpha
)=p(\varphi (\alpha ))$ holds. The first step in this translation is the
passage from a probability function $\pi $ to a valuation on $\mathcal{F}%
_{n}/\sim $, the Lindenbaum algebra of \ formulas in $L_{n}$.

We shortly recall the construction of the Lindenbaum algebra of formulas.
The set $2=\{0,1\}$ is the set of the classical truth values and $2^{n}$ is
the set of all possible worlds (truth-value assignments to the variables in $%
P_{n}$). We can attribute a meaning , conceived as the set of possible
worlds in which $\alpha $ holds true, to every formula $\alpha $: this is
the task of a function $M$ that takes $F_{n}$ into the field of sets $%
\mathcal{P}(2^{n})=(P(2^{n}),\cap ,\cup ,-,\emptyset ,2^{n})$. The inductive
definition of $M$ is as follows: 
\begin{eqnarray*}
M(p_{i}) &=&\{s\in 2^{n}:s(i)=1\} \\
M(\alpha \wedge \beta ) &=&M(\alpha )\cap M(\beta ), \\
M(\alpha \vee \beta ) &=&M(\alpha )\cup M(\beta ), \\
M(\lnot \alpha ) &=&-M(\alpha ), \\
M(0) &=&(\emptyset ), \\
M(1) &=&2^{n}.
\end{eqnarray*}

\noindent If we denote with $\mathcal{F}_{n}$ be the algebra of formulas of $%
L_{n}$, i.e. the absolutely free algebra on the generators $%
P_{n}=\{p_{1},...,p_{n}\}$, then $M$ is the unique extension of the function 
$j(p_{i})=\{s\in 2^{n}:s(i)=1\}$ from $P_{n}$ to $\mathcal{P}(2^{n})$, to a
homomorphism from $\mathcal{F}_{n}$ to $\mathcal{P}(2^{n})$.  We define a
congruence relation on formulas setting $\alpha \sim \beta $ iff $\alpha
\models \beta $ and $\beta \models \alpha $ iff $M(\alpha )=M(\beta )$,
collecting in the same block formulas with identical meaning. The Lindenbaum
algebra of $L_{n}$ is $\mathcal{F}_{n}/\sim $. \ The top of the algebra is $%
1=|1|$, the set of all tautologies, the bottom is $0=|0|$, the set of all
contradictions. As every subset $X\subseteq 2^{n}$ can be defined by an $n$%
-ary formula $\alpha $, $M$ is a surjective function, so there is an
isomorphism $\psi $ from $\mathcal{F}_{n}/\sim $ to $\mathcal{P}(2^{n})$
defined by $\psi (|\alpha |)=M(\alpha )$. When $L$ has numerably many variables, 
$\mathcal{F}/\sim $ is
isomorphic to a proper subalgebra of $\mathcal{P}(2^{\omega })$.

\begin{theorem}
\label{teo224}If $\pi $ is a probability function on $L_{n}$, then $\pi
^{\ast }:\mathcal{F}_{n}/\!\sim \ \rightarrow \lbrack 0,1]$ defined by $\pi
^{\ast }(|\alpha |)=\pi (\alpha )$ is a Boolean valuation on $\mathcal{F}%
_{n}/\!\sim $
\end{theorem}

\begin{proof}
Firstly, we verify that $\pi ^{\ast }$ is well-defined on the equivalence
classes. In fact, if $|\alpha |=|\beta |$ then $\alpha \sim \beta $ and then 
$\pi (\alpha )=\pi (\beta )$, by point 2) of theorem \ref{teoprob}, and so $%
\pi ^{\ast }(|\alpha |)=\pi ^{\ast }(|\beta |)$. Now we have only to show
that $p^{\ast }$ is a Boolean valuation. If $|\alpha |=1$ then $|\alpha |=1$
in the Lindenbaum algebra and so $\models \alpha $ and \ then $\pi (\alpha
)=1$, because $\pi $ is a probability function. Then $\pi ^{\ast }(|\alpha
|)=1$ by definition of $\pi ^{\ast }$. If $|\alpha |=0$ then $|\lnot \alpha
|=1$ and $\pi (\lnot \alpha )=1$ so $\pi (\alpha )=0$ by point 1) of theorem 
\ref{teoprob}. Then $\pi ^{\ast }(|\alpha |)=0$ by definition of $\pi ^{\ast
}$. \ Finally 
\begin{eqnarray*}
\pi ^{\ast }(|\alpha |\vee |\beta |)&=&\pi ^{\ast }(|\alpha \vee \beta |) \\
&=&\pi (\alpha \vee \beta ) \\
&=&\pi (\alpha )+\pi (\beta )-\pi (\alpha \wedge \beta ) \\
&=&\pi ^{\ast }(|\alpha |)+\pi ^{\ast }(|\beta |)-\pi ^{\ast }(\alpha \wedge
\beta ),
\end{eqnarray*}
\noindent where the third line follows from point 4) of theorem \ref{teoprob}.
\end{proof}

\begin{theorem}
\label{teotrad}If $\pi $ is probability function on $L_{n}$, then there is a
probability space $(A,\mathcal{P}(A),p)$, where $A=2^{n}$, and a morphism $%
\varphi :\mathcal{F}_{n}\rightarrow \mathcal{P}(A)$ such that, for all $%
\alpha \in L_{n}$, $\pi (\alpha )=p(\varphi (\alpha ))$.
\end{theorem}

\begin{proof}
By the preceding theorem, we can extract from $\pi $ a valuation $\pi ^{\ast
}$ on $\mathcal{F}_{n}/\!\sim $. We set $A=2^{n}$ and define, for all $%
X\subseteq A$, 
\begin{equation*}
p(X)=\pi ^{\ast }(\psi ^{-1}(X)),
\end{equation*}
where $\psi $ is the isomorphism from $\mathcal{F}_{n}/\!\sim $ to $\mathcal{%
P}(A)$ defined by $\psi (|\alpha |)=M(\alpha )$. 
So $p$, by theorem \ref{teovalind}, is the valuation induced on $\mathcal{P%
}(A)$ by the valued Boolean algebra $(\mathcal{F}_{n}/\!\sim ,\pi ^{\ast })$
and by the isomorphism $\psi $. As $p$ is a Boolean valuation, it satisfies
Kolmogorov's axioms \ by theorem \ref{teoadd} and then $(A,\mathcal{P}(A),p)$
is a probability space. Now we set $\varphi =\psi \circ |x|$ and it can be
immediately seen that $\varphi $ is a morphism, resulting from the
composition of two morphisms, and $\pi (\alpha )=p(\varphi (\alpha ))$.
Finally, we have 
\begin{equation*}
p(\varphi (\alpha ))=p(\psi (|\alpha |))=\pi ^{\ast }(\psi ^{-1}(\psi
(|\alpha |)))=\pi ^{\ast }(|\alpha |),
\end{equation*}
where $\pi (\alpha )=\pi ^{\ast }(|\alpha |)$ by the preceding theorem
\end{proof}

We remark that, in the above theorem, $\varphi$ is $M$, the function 
taking every formula to its meaning. The following diagram shows the
passage from $\pi $ to $p$.

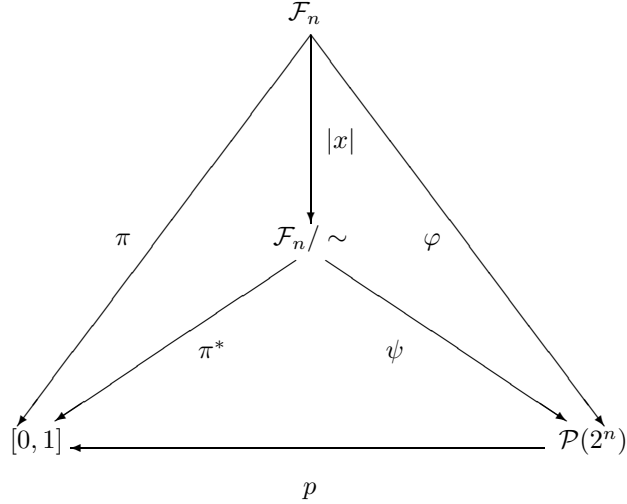
\begin{center}
\setlength{\unitlength}{1mm} 
\begin{picture}(75,65)

\put(71,6){\vector(-1,0){63}}
\put(40,61){\vector(3,-4){39}}
\put(40,61){\vector(-3,-4){39}}

\put(40,61){\vector(0,-1){25}}
\put(42,31){\vector(3,-2){32}}
\put(38,31){\vector(-3,-2){32}}

\put(35,33){$\mathcal{F}_{n}/\sim$}
\put(73,6){$\mathcal{P}(2^{n})$}
\put(37,63){$\mathcal{F}_{n}$}
\put(0,6){$[0,1]$}

\put(42,46){$|x|$}
\put(55,33){$\varphi$}
\put(50,18){$\psi$}
\put(39,0){$p$}
\put(25,18){$\pi^\ast$}
\put(14,33){$\pi$}

\end{picture}
\end{center}

As a consequence of this result we can derive theorems about probability
over sentences from theorems about probability ever sets. Consider for
instance the following theorem about $\pi $: $\alpha \models \beta $ implies 
$\pi (\alpha )\leq \pi (\beta )$. If we take the probability space
associated to $\pi $, where the algebra of events is the algebra of meanings
of formulas, we have that $\alpha \models \beta $ implies $M(\alpha
)\subseteq M(\beta )$, but we know that in every probability space $%
X\subseteq Y$ implies $p(X)\subseteq p(Y)$, so we have $p(M(\alpha )\leq
p(M(\beta ))$. If $p$, $\pi ^{\ast }$ and $\psi $ are as in the diagram
above, and remembering that $\varphi $ is the `meaning' function $M$, we
have 
\begin{equation*}
p(M(\alpha ))=\pi ^{\ast }(\psi ^{-1}(M(\alpha )))=\pi ^{\ast }(|\alpha
|)=\pi (\alpha ).
\end{equation*}
In the same way we get $p(M(\beta )=\pi (\beta )$, so we can conclude with $%
\pi (\alpha )\leq \pi (\beta )$.

Now we face the problem of translating a probability measure over sets into
a probability function over formulas. The idea behind this translation is 
that every event $X$ can be seen as the meaning $M(\alpha )$ of a 
formula $\alpha $, so we could define a probability function $\pi$ 
and a translation  $\tau :P(A)\rightarrow F$ setting $\tau(X)=\alpha$, in 
order to obtain, for every event $X \subseteq A$,  $p(X)=\pi (\tau (X))$
The problem with this choice of $\tau$
is that it is not univocal: if $\xi$ is any formula logically equivalent to $\alpha$, then
$M(\alpha)=M(\xi)$ and we could as well set $\tau(X)=\xi$. So the best we can do is to look 
for a probability function $\pi$ and a function from events 
to set of formulas, i.e. a $\overline{\tau }%
:P(A)\rightarrow P(F)$ such that:
\begin{equation*}
\text{for all }X\subseteq A\text{, for all }\alpha \in \overline{\tau }(X)%
\text{, }p(X)=\pi (\alpha ),
\end{equation*}
where formulas in $\overline{\tau }(X)$ are all logically equivalent.

We suppose that $(A,\mathcal{P}(A),p)$ is a probability space and we start
considering the case where the sample space $A$ is $2^{n}$. We know that
there is an isomorphism $\psi $ from $\mathcal{F}_{n}/\!\sim $\ to $\mathcal{%
P}(2^{n})$, given by $\psi (|\alpha |)=M(\alpha )$, and then we define $%
\overline{\varphi }=\psi ^{-1}$. As every $X\subseteq 2^{n}$ is $M(\alpha )$
for some $\alpha $, $\overline{\varphi }$ is a function from $P(A)$ to $P(F)$%
. Now we define the probability function $\pi $: \ by theorem \ref{teovalind}%
, we know that a valuation $v=p\cdot \psi $\ is induced on $\mathcal{F}%
_{n}/\!\sim $ by $(\mathcal{P}(2^{n}),p)$ and $\psi ^{-1}$. The translation
is completed by proving that a probability function can be recovered from a
Boolean valuation on a Lindenbaum algebra.

\begin{theorem}
\label{teo225}If $v$ is a Boolean valuation on $\mathcal{F}_{n}/\sim $, then
the function $\pi $ from $F_{n}$ to$[0,1]$ defined by $\pi (\alpha
)=v(|\alpha |)$ is a probability function on $L_{n}$.
\end{theorem}

\begin{proof}
Firstly we observe that $0\leq \pi (\alpha )\leq 1$, as $v $ is a bounded
lattices valuation. Then we verify the first axiom: suppose that $\models
\alpha $, then $\models \alpha \leftrightarrow \delta $ and so $|\alpha
|=|\delta |=1$ in $\mathcal{F}_{n}/\sim $. As $v $ is a valuation, $v
(|\alpha |)=1$ and so, by definition of $\pi $, $\pi (\alpha )=1$. Finally
we verify the second axiom: suppose that $\models \lnot (\alpha \wedge \beta
)$, then in $\mathcal{F}_{n}/\sim $ we have $\lnot |\alpha \wedge \beta |=1$
and $|\alpha |\wedge |\beta |=0$. Then 
\begin{equation*}
\pi (\alpha \vee \beta )=v(|\alpha \vee \beta |) 
=v(|\alpha |\vee |\beta |)
=v(|\alpha |)+v(|\beta |) 
=\pi (\alpha )+\pi (\beta ),
\end{equation*}
where the third line follows because $v $, as a Boolean valuation, is
additive by theorem \ref{teoadd}.
\end{proof}

So we can conclude that, given a probability space $(A,\mathcal{P}(A),p)$
where $A=2^{n}$, there is a probability function $\pi $ on $L_{n}$ and a
function $\overline{\tau }:P(A)\rightarrow P(F)$ such that, for all $%
\alpha \in \overline{\tau }(X)$, $p(X)=\pi (\alpha )$ holds. We only set $%
\overline{\tau }=\psi ^{-1}$ and apply the preceding theorem setting $\pi
(\alpha )=v(|\alpha |)$, where $v$ is the valuation induced on $\mathcal{F}%
_{n}/\!\!\sim $ by $\psi ^{-1}$ and the valued Boolean algebra $(\mathcal{P}%
(A),p)$.

This argument is grounded on the assumption $A=2^{n}$, that gives the
isomorphism between $\mathcal{P}(A)$ and $\mathcal{F}_{n}/\sim $. The
general case of an algebra of events $\mathcal{P}(A)$, where $A$ is any
finite set, follows by a slightly different argument. We start observing
that we can identify such algebras with the Boolean algebras $\mathbf{2}^{n}$%
, as $n$ varies on the natural numbers. As $\mathcal{F}_{n}/\sim $ is
isomorphic to $\mathcal{P}(2^{n})$ and then isomorphic to $\mathbf{2}%
^{2^{n}} $, the finite Lindenbaum algebras are exactly the algebras of kind $%
\mathbf{2}^{n}$ where $n$ is a power of $2$. Before proving the translatability of
probability on sets in terms of probability on sentences, we need the
following lemma about the existence of epimorphisms between finite Boolean
algebras.

\begin{theorem}
For every $n\leq k$,

\begin{enumerate}
\item  there is an epimorphism $\theta :\mathbf{2}^{k}\rightarrow \mathbf{2}%
^{n}$,

\item  there is a sequence $s\in 2^{k}$ such that $[0,s]$ is isomorphic to $%
\mathbf{2}^{n}$.
\end{enumerate}
\end{theorem}

\begin{proof}
1. For every $s\in 2^{k}$, we define $\theta (s)=s\upharpoonright
n=(s_{0},...,s_{n-1})$. $\theta $ is clearly onto $\mathbf{2}^{n}$. If we
denote with $0^{m}$ a $m$-termed sequence of \ $0$, we have that, for all $%
q\in 2^{n}$, the sequence $q\ast 0^{k-n}$ belongs to $2^{k}$ and $\theta
(q\ast 0^{k-n})=q$. We prove that $\theta $ is a Boolean morphism: 
\begin{eqnarray*}
\theta (\lnot s)=(\lnot s)\upharpoonright n=\lnot (s\upharpoonright n)=\lnot
\theta (s), \\
\theta (s\wedge q)=(s\wedge q)\upharpoonright n=(s\upharpoonright n)\wedge
(q\upharpoonright n)=\theta (s)\wedge \theta (q).
\end{eqnarray*}
The preservation of $\vee $, $0$ and $1$ is an easy consequence.

2. Let $s=1^{n}\ast 0^{k-n}$: one sees immediately that $\theta $ is a
bijection between $[0,s]$ and $2^{n}$.
\end{proof}

Now we are ready for the translation of a generic (finite) probability space.

\begin{theorem}
Let $(A,\mathcal{P}(A),p)$ be a probability space where $A=\{a_{1},...,a_{n}%
\}$ and let $k=\min x(n\leq 2^{x})$. There is a function $\overline{\tau }
$:$P(A)\rightarrow P(F)$ and a probability function $\pi $ on $L_{k}$ such
that, for all event $X\subseteq A$ and all formula $\alpha \in \overline{%
\tau }(X)$, $p(X)=\pi (\alpha )$ holds.
\end{theorem}

\begin{proof}
We define three morphisms as follows.

\begin{enumerate}
\item  As $A=\{a_{1},...,a_{n}\}$, the algebra of events $\mathcal{P}(A)$ is
isomorphic to $\mathbf{2}^{n}$, so there is an isomorphism $\chi :\mathbf{2}%
^{n}\rightarrow \mathcal{P}(A)$.

\item  In general $\mathbf{2}^{n}$ is not isomorphic to a Lindenbaum
algebra, so we let $k=\min x(n\leq 2^{x})$ in order to have 
\begin{equation*}
\mathbf{2}^{2^{k}}\simeq P(\mathbf{2}^{k})\simeq \mathcal{F}_{k}/\!\sim .
\end{equation*}
\ By point 1) of the above theorem, there is a morphism $\theta :\mathbf{2}%
^{2^{k}}\rightarrow \mathbf{2}^{n}$.

\item  We know that there is an isomorphism $\psi $ from $\mathcal{F}%
_{k}/\!\sim $ to$\ \mathbf{2}^{2^{k}}$, defined by $\psi (|\alpha
|)=M(\alpha )$. Then we define the morphism $\eta =\chi \circ \theta \circ
\psi $ from $\mathcal{F}_{k}/\sim $ to $\mathcal{P}(A)$, as in the following
diagram.
\end{enumerate}

\begin{center}
\setlength{\unitlength}{1mm} 
\begin{picture}(65,35)

\put(10,26){\vector(1,0){14}}
\put(28,26){\vector(1,0){14}}
\put(45,26){\vector(1,0){14}}

\put(5,24){\vector(0,-1){20}}

\put(58,25){\vector(-2,-1){48}}

\put(0,26){$\mathcal{F}_k/\sim$}
\put(25,26){$2^{2^{k}}$}
\put(43,26){$2^{n}$}
\put(60,26){$\mathcal{P}(A)$}
\put(2,0){$[0,1]$}

\put(15,28){$\psi$}
\put(33,28){$\theta$}
\put(50,28){$\chi$}

\put(7,14){$v$}
\put(36,11){$p$}
%\put(20,10){$\eta=\chi\circ\theta\circ\psi$}

\end{picture}
\end{center}
Now we can define $\overline{\tau }:P(A)\rightarrow P(F)$ setting 
\begin{equation*}
\overline{\varphi }(X)=\bigcup \{|\xi |\in \mathcal{F}_{k}/\!\sim :\eta
(|\xi |)=X\}.
\end{equation*}
As for the probability function, we observe that, by theorem \ref{teo222}, $%
v=p\circ \eta $ is a bounded lattices valuation and then a Boolean valuation
on $\mathcal{F}_{k}/\!\sim $. By theorem \ref{teo225}, we obtain from $v $
a probability function $\pi $ on $L_{k}$ such that $\pi (\alpha )=v(|\alpha
|)$. The functions $\pi $ and $\overline{\tau }$ are as required by the
theorem, because for all $X\subseteq A$ and all $\alpha \in \overline{%
\tau }(X)$, 
\begin{equation*}
\pi (\alpha )=v(|\alpha |)=p(\eta (|\alpha |))=p(X),
\end{equation*}
where the last equation follows from $\alpha \in \overline{\tau }(X)$.
\end{proof}

\section{Partial probability spaces\label{parprobpar}}

If we understand probability as the measure of an event-set, then an
essential role is played by Boolean algebras and Boolean valuations. We
start introducing the concept of partial event and a measure of
probability on partial events, then we show that the algebraic counterparts
of these concepts are given by DMF-algebras and partial valuations on
DMF-algebras.

Partial events have been introduced elsewhere (see \cite{negri2010}). In order to
make this work self-contained, we recall the definitions of the main
concepts.

Given set $S$, that can be conceived as the sample space of an experiment,
we define the set of all \textit{partial sets on} $S$ as the set $%
D(S)=\{(A,B):A,B\subseteq S,$ and $A\cap B=\emptyset \}$. \ The elements of $%
A$ ($B$) are the \textit{positive (negative) elements} of the partial set $%
(A,B)$. Partial sets can be given an algebraic structure with the following
operations: 
\begin{eqnarray*}
(A,B)\sqcap (C,D) &=&(A\cap C,B\cup D), \\
(A,B)\sqcup (C,D) &=&(A\cup C,B\cap D), \\
-(A,B) &=&(B,A), \\
0 &=&(\emptyset ,S), \\
1 &=&(S,\emptyset ), \\
n &=&(\emptyset ,\emptyset ).
\end{eqnarray*}
We define a binary relation between partial sets setting 
\begin{equation*}
(A,B)\sqsubseteq (C,D)\text{ iff }A\subseteq C\text{ and }D\subseteq B\text{.%
}
\end{equation*}

\noindent It can be easily proved that $(D(S),\sqsubseteq )$ is a partially
ordered set having a maximum $(S,\emptyset )$ and a minimum $(\emptyset ,S)$.

\noindent The algebra $\mathcal{D}(S)=(D(S),\sqcap ,\sqcup ,-,0,n,1)$ is the 
\textit{algebra of all partial sets} on $S$ or the \textit{algebra of
partial events} on $S$. A\textit{\ field of partial set} is any subalgebra
of $\mathcal{D}(S)$. A partial set $(A,B)$ is a \textit{Boolean partial set}
if $A=S-B$. The set of Boolean partial sets\ is a Boolean algebra isomorphic
to $\mathcal{P}(S)$, the classical power-set algebra.

When $S$ is a sample space, we say that $\mathcal{D}(S)$ is the \textit{%
algebra of partial events} on $S$. Every positive (negative) element of $%
(A,B)$ is a favorable (unfavorable) case of the event. In relation with the
experimental result $s\in S$, we say that $(A,B)$ occurs positively if $s\in
A$, occurs negatively if $s\in B$, is indeterminate otherwise. Events of
classical probability theory are to be identified with Boolean partial
sets.

The probability value of a partial event $(A,B)$ is a pair $(x,y)$ belonging
to the set $T$ of \textit{partial probability values}, defined by $%
T=\{(x,y)\in \lbrack 0,1]^{2}:$ $x+y\leq 1\}$. We define the relation $%
\preceq $ on $T$ setting $(x,y)\preceq (w,z)$ iff $x\leq w$ e $z\leq y$. It
can be easily shown that $(T\preceq )$ is a poset with a maximum $(1,0)$ and
a minimum $(0,1)$.

Given a partial field of set $\mathcal{G}_{S}$ on $S$ and a function $\mu
:G_{S}\rightarrow T$, we say that $\mu $ is a \textit{measure of partial
probability} \ when the following axioms are satisfied:

\begin{enumerate}
\item  $\mu (S,\emptyset )=(1,0)$.

\item  $\mu (A,B)+\mu (C,D)=\mu ((A,B)\sqcup (C,D))-\mu ((A,B)\sqcap (C,D)).$

\item  $\mu (-(A,B))=\sigma (\mu (A,B))$,

\item  $(0,0)\preceq \mu (A,\emptyset )$, for all $(A,\emptyset )\in \nabla $%
,
\end{enumerate}

\noindent where $\sigma :[0,1]^{2}\rightarrow \lbrack 0,1]^{2}$ is defined
by $\sigma (x,y)=(y,x)$. (The definition in par. 3 of \cite{negri2010} slightly
differs in axiom 2.) A \textit{partial probability space} is a triple $(S,%
\mathcal{G}_{S},\mu )$ where $\mathcal{G}_{S}$ is a field of partial sets on
the sample space $S$ and $\mu $ is a measure of partial probability.

A partial probability space can be easily obtained from every classical
probability space $(S,\mathit{P}(S),p)$ as follows: we define the \textit{%
partial probability space associated to} $(S,\mathit{P}(S),p)$ as $(S,%
\mathcal{D}(S),\mu )$, where $\mu :D(S)\rightarrow T$ is defined by 
\begin{equation*}
\mu (A,B)=(p(A),p(B)).
\end{equation*}
As $(A,B)$ is a partial event, $A\cap B=\emptyset $ so $p(A)+p(B)=p(A\cup B)$%
: this proves that $\mu (A,B)\in T$. It can be easily proved that $\mu $
satisfies the four \ axioms above, so $(S,\mathcal{D}(S),\mu )$ is a partial
probability space.

Instead of proving some properties of partial probability spaces, we turn to
the algebraic counterparts of these concepts in order to prove similar
results in a much more general setting.

\section{DMF-algebras}

We introduce DMF-algebras as the algebraic counterparts of partial sets. A
De Morgan algebra (DM-algebra) is an algebra of type $\mathcal{L}%
_{DM}=\{\wedge ,\vee ,\lnot ,0,1\}$ that satisfies, besides the axioms of
bounded distributive lattices, the following axioms: 
\begin{eqnarray*}
\lnot \lnot x &=&x\text{ (double negation)}, \\
\lnot (x\wedge y) &=&\lnot x\vee \lnot y\text{ (de Morgan law).}
\end{eqnarray*}
(The other de Morgan law easily follows.) We remember that in every
DM-algebra $\lnot 0=1$ and $\lnot 1=0$. A DMF-algebra is a DM-algebra with a
single fixed point for negation. The existence of a single fixed point can
be obtained equationally, at the price of extending the type $\mathcal{L}$
with the constant $n$, by the following axioms: 
\begin{eqnarray*}
x\wedge \lnot x &\leq &y\vee \lnot y\text{ (normality)} \\
\lnot n &=&n\text{ (fixed point)}
\end{eqnarray*}

We denote with $DMF$ the class of all DMF-algebras. In every DMF-algebra $%
\mathcal{A}$ we denote with $\nabla ^{\mathcal{A}}$ the 
set of all elements of type $x\vee \lnot x$. We
simply write $\nabla $ when no confusion is possible. One can easily see
that $\nabla =[n,1]$: on one side we have $n=\lnot n\wedge n\leq x\vee \lnot
x$ for all $x$, on the other side, if $n\leq x\leq 1$ then $0\leq \lnot
x\leq n$ and then $x=x\vee \lnot x.$ We can dually define $\Delta $ as the
set of all elements of type $x\wedge \lnot x$, and prove that $\Delta =[0,n]$%
. We sometimes abbreviate $x\vee \lnot x$ with $\nabla (x)$ and $x\wedge
\lnot x$ with $\Delta (x)$.

We denote with $K$ the set of all complemented elements of a DMF-algebra $%
\mathcal{A}$. Members of $K$ are also said \textit{\ Boolean elements}. If
the complement of $a$ exists, we denote it by $a^{\ast }$.

\begin{theorem}
\label{teoBA}If $\mathcal{A}$ is a DMF-algebra,

\begin{enumerate}
\item  $K$ is a Boolean algebra,

\item  $K=\{x:x\vee \lnot x=1\}$,

\item  for all $a\in K$, $\lnot a=a^{\ast }$.
\end{enumerate}
\end{theorem}

\begin{proof}
1. For all $a$, $b\in K$ we have $(a\wedge b)^{\ast }=(a^{\ast }\vee b^{\ast
})$, because $(a\wedge b)\vee (a^{\ast }\vee b^{\ast })=1$ and $(a\wedge
b)\wedge (a^{\ast }\vee b^{\ast })=0$. So $a\wedge b$ is complemented and
belongs to $K$. In the same way we prove that $a\vee b$ is complemented and
belongs to $K$, thus $K$ is closed with respect to $\wedge $ and $\vee $.
Obviously $0$ and $1$ are in $K$, so $K$ is a complemented distributive
bounded lattice.

2. Firstly we show that, for all$\ a\in K$, $\lnot (a^{\ast })=(\lnot
a)^{\ast }$. From $a\vee a^{\ast }=1$ and $a\wedge a^{\ast }=0$ we have $%
\lnot a\wedge \lnot (a^{\ast })=0$ and $\lnot $ $a\vee \lnot (a^{\ast })=1$,
thus showing that $\lnot (a^{\ast })$ is the complement of $\lnot a$. \ Now
we can prove that, if $a\in K$ then 
\begin{equation*}
1=(a\vee \lnot a)\vee (a^{\ast }\wedge \lnot (a^{\ast }))=a\vee \lnot a,
\end{equation*}
where the last equation follows by the normality axiom. In the other
direction, it is immediate to show that if $a\vee \lnot a=1$ then $a$ is
complemented and belongs to $K$.

3. For all $a\in K$ we have both $a\vee a^{\ast }=1$ by definition and $%
a\vee \lnot a=1$ by point 2, thus $a\vee a^{\ast }=$ $a\vee \lnot a$. In the
same way we get $a\wedge a^{\ast }=$ $a\wedge \lnot a$ and, by
distributivity, $\lnot a=a^{\ast }$.
\end{proof}

The generation of partial sets as disjoint pairs of classical sets is the
leading idea of a general construction of DMF-algebras from bounded
distributive lattices. We denote with $DL_{0,1}$ the class of bounded
distributive lattices. If $\ \mathcal{A\in }DL_{0,1}$ then $\mathcal{A}%
\times \mathcal{A}^{\circ }\in DL_{0,1}$, where $\mathcal{A}^{\circ }$
denotes the dual of $\mathcal{A}$. We define 
\begin{equation*}
\pi (A)=A\times A^{\circ }\upharpoonright \{(a,b):a\wedge b=0\}.
\end{equation*}
Finally we define a structure $\pi (\mathcal{A})$ of type $\mathcal{L}_{DMF}=%
\mathcal{L}_{DM}\cup \{n\}$ setting $\pi (\mathcal{A})=(\pi (A),\wedge ^{\pi
(\mathcal{A})},\vee ^{\pi (\mathcal{A})},\lnot ^{\pi (\mathcal{A})},0^{\pi (%
\mathcal{A})},1^{\mathcal{\pi (A)}},n^{\pi (\mathcal{A})})$ where 
\begin{eqnarray*}
(x,y)\wedge ^{\pi (\mathcal{A})}(z,w) &=&(x\wedge z),(y\wedge ^{\circ
}w)=(x\wedge z),(y\vee w) \\
(x,y)\vee ^{\pi (\mathcal{A})}(z,w) &=&(x\vee z),(y\vee ^{\circ }w)=(x\vee
z),(y\wedge w) \\
\lnot ^{\pi (\mathcal{A})}(x,y) &=&(y,x) \\
0^{\pi (\mathcal{A})} &=&(0,0^{\circ })=(0,1) \\
1^{\pi (\mathcal{A})} &=&(1,1^{\circ })=(1,0) \\
n^{\pi (\mathcal{A})} &=&(0,0).
\end{eqnarray*}

\noindent As in every lattice, we can introduce a partial order in $\pi (%
\mathcal{A})$ setting $(a,b)\leq ^{\pi (\mathcal{A})}(c,d)$ iff $(a,b)\wedge
^{\pi (\mathcal{A})}(c,d)=(a,b)$. By an easy calculation we obtain 
\begin{equation*}
(a,b)\leq ^{\pi (\mathcal{A})}(c,d)\text{ iff }a\leq c\text{ and }d\leq b%
\text{.}
\end{equation*}
The following theorem shows that $\pi (-)$ is an uniform way of constructing
DMF-algebras from bounded distributive lattices.

\begin{theorem}
\label{teopi}If $\mathcal{A}\in DL_{0,1}$ then $\pi (\mathcal{A})\in DMF$.
\end{theorem}

\begin{proof}
Firstly we show that $\pi (A)$ is closed with respect to $\wedge ^{\pi (%
\mathcal{A})}$ and $\vee ^{\pi (\mathcal{A})}$. If $(a,b)$ and $(c,d)$ are
in $\pi (A)$, then $a\wedge b=c\wedge d=0$, so
$(a,b)\wedge ^{\pi (\mathcal{A})}(c,d)=(a\wedge c,b\vee d)$
where 
$(a\wedge c)\wedge (b\vee d)=(a\wedge c\wedge b)\vee (a\wedge c\wedge d)=0$.
The same holds for $\vee ^{\pi (\mathcal{A})}$. As $\pi (A)$ is closed with
respect to $0^{\vee ^{\pi (\mathcal{A})}}=(0,1)$ and $1^{\vee ^{\pi (%
\mathcal{A})}}=(1,0)$, we can say that $\pi (\mathcal{A})$, as a subalgebra
of $\mathcal{A}\times \mathcal{A}^{\circ }$, is a bounded distributive
lattice (forgetting $\lnot $ and $n$). It can be easily shown that $\pi (%
\mathcal{A})$ is a DM-algebra: double negation law holds because 
$\lnot ^{\pi (\mathcal{A})}\lnot ^{\pi (\mathcal{A})}(a,b)=(a,b)$
and de Morgan law holds because 
\begin{eqnarray*}
\lnot ^{\pi (\mathcal{A})}((a,b)\wedge ^{\pi (\mathcal{A})}(c,d))&=&\lnot
^{\pi (\mathcal{A})}(a\wedge c,b\vee d) \\
&=&(b\vee d,a\wedge c) \\
&=&(b,d)\vee ^{\pi (\mathcal{A})}(a,c) \\
&=&\lnot ^{\pi (\mathcal{A})}(a,b)\vee ^{\pi (\mathcal{A})}\lnot ^{\pi (%
\mathcal{A})}(c,d).
\end{eqnarray*}
Finally we show that $\pi (\mathcal{A})$ is a DMF-algebra. In the first
place, we observe that $n^{\pi (\mathcal{A})}$ is a fixed point of $\lnot $,
because $\lnot ^{\pi (\mathcal{A})}n^{\pi (\mathcal{A})}=\lnot ^{\pi (%
\mathcal{A})}(0,0)=n^{\pi (\mathcal{A})}$. Then we show normality: 
\begin{equation*}
(a,b)\wedge ^{\pi (\mathcal{A})}\lnot ^{\pi (\mathcal{A})}(a,b)\leq ^{\pi (%
\mathcal{A})}(c,d)\vee ^{\pi (\mathcal{A})}\lnot ^{\pi (\mathcal{A})}(c,d).
\end{equation*}
This amounts to prove $(a,b)\wedge ^{\pi (\mathcal{A})}(b,a)\leq ^{\pi (%
\mathcal{A})}(c,d)\vee ^{\pi (\mathcal{A})}(d,c)$, i.e. $(a\wedge b,b\vee
a)\leq ^{\pi (\mathcal{A})}(c\vee d,d\wedge c)$, what follows immediately
remembering that $a\wedge b=c\wedge d=0$.
\end{proof}

From now on we adopt a less baroque notation and write simply $(a,b)\wedge
(c,d)=(a\wedge c,b\vee d)$ instead of $(a,b)\wedge ^{\pi (\mathcal{A}%
)}(c,d)=(a\wedge c,b\vee d)$, leaving to the reader the task of
distinguishing between $\wedge $ as an operator on pairs and $\wedge $ as an
operator on individuals. The same holds for $\vee $, $\lnot $, $0$, $1$, $n$
and $\leq $.

The following theorem shows that every DMF-algebra $\mathcal{A}$ can be
obtained, through the construction $\pi (-)$, as a subalgebra of $\pi (\nabla
_{\mathcal{A}})$.

\begin{theorem}
\label{teocostrdmf}If $\mathcal{A}\in DMF$ then there is a monomorphism $%
\varphi :\mathcal{A}\rightarrow \pi (\nabla _{\mathcal{A}})$.
\end{theorem}

\begin{proof}
We define $\varphi :\mathcal{A}\rightarrow \pi (\nabla _{\mathcal{A}})$
setting $\varphi (x)=(x\vee n,\lnot x\vee n)$. Firstly we show that $\varphi
(x)\in \pi (\nabla _{\mathcal{A}})$. We observe that $\nabla _{\mathcal{A}}$
is considered as a bounded lattice, so $1^{\nabla _{\mathcal{A}}}=1$ and $%
0^{\nabla _{\mathcal{A}}}=n$. Meet and join of $\nabla _{\mathcal{A}}$ are
inherited from $\mathcal{A}$, so we simply write $\wedge $ and $\vee $
instead of $\wedge ^{\nabla _{\mathcal{A}}}$ and $\vee ^{\nabla _{\mathcal{A}%
}}$ when no confusion is possible. Both $x\vee n$ and $\lnot x\vee n$
belongs to $\nabla _{\mathcal{A}}$, because $n\leq x\vee n$ and $n\leq \lnot
x\vee n$, so we must only verify that
$(x\vee n)\wedge (\lnot x\vee n)=0^{\nabla _{\mathcal{A}}}$,
what follows immediately because $(x\vee n)\wedge (\lnot x\vee n)=(x\wedge
\lnot x)\vee n=n.$

$\varphi $ is injective, for suppose $\varphi
(x)=\varphi (y)$, then $x\vee n=y\vee n$ and $\lnot x\vee n=\lnot y\vee n$
and so also $x\wedge n=y\wedge n$. From distributivity $x=y$ follows.

$\varphi $ preserves $\wedge $: 
\begin{eqnarray*}
\varphi (x\wedge y)&=&((x\wedge y)\vee n,\lnot (x\wedge y)\vee n) \\
&=&((x\wedge y)\vee n,\lnot x\vee \lnot y\vee n) \\
&=&((x\vee n)\wedge (y\vee n),(\lnot x\vee n)\vee (\lnot y\vee n)) \\
&=&(x\vee n,\lnot x\vee n)\wedge ^{\pi (\mathcal{A})}(y\vee n,\lnot y\vee n) \\
&=&\varphi (x)\wedge ^{\pi (\mathcal{A})}\varphi (y).
\end{eqnarray*}
An analogous proof shows that $\varphi $ preserves $\vee $. As for $\lnot $: 
\begin{equation*}
\varphi (\lnot x)=(\lnot x\vee n,\lnot \lnot x\vee n)
=(\lnot x\vee n,x\vee n)
=\lnot ^{\pi (\mathcal{A})}(x\vee n,\lnot x\vee n)
=\lnot {^{\pi (\mathcal{A})}}(\varphi (x)).
\end{equation*}
$\varphi $ preserves $0$, $1$ and $n$: 
\begin{eqnarray*}
\varphi (0)=(n,1)=(0^{\nabla _{\mathcal{A}}},1^{\nabla _{\mathcal{A}%
}})=0^{\pi (\mathcal{A})}, \\
\varphi (0)=(1,n)=(1^{\nabla _{\mathcal{A}}},0^{\nabla _{\mathcal{A}%
}})=1^{\pi (\mathcal{A})}, \\
\varphi (0)=(n,n)=(0^{\nabla _{\mathcal{A}}},0^{\nabla _{\mathcal{A}%
}})=n^{\pi (\mathcal{A})}.
\end{eqnarray*}
\end{proof}

For every $a\in A$, we call $a\vee n$ its \textit{positive part} and $\lnot
a\vee n$ its \textit{negative part}, what deliberately recalls the positive
and negative elements of a partial set $(A,B)$. In fact, when $a$ is a
partial set $(A,B)$, we have $(A,B)\sqcup n=(A,\emptyset )$ and $%
-(A,B)\sqcup n=(B,\emptyset )$, where $A$ and $B$ are respectively the set
of the positive and negative elements of $(A,B)$. In conclusion, every $a\in
A$, as a consequence of the above theorem, can be seen as a pair constituted
by its positive part $a\vee n$ and its negative part $\lnot a\vee n$.

The following picture shows an example of $\mathcal{A}$ and $\pi (\nabla _{%
\mathcal{A}})$.

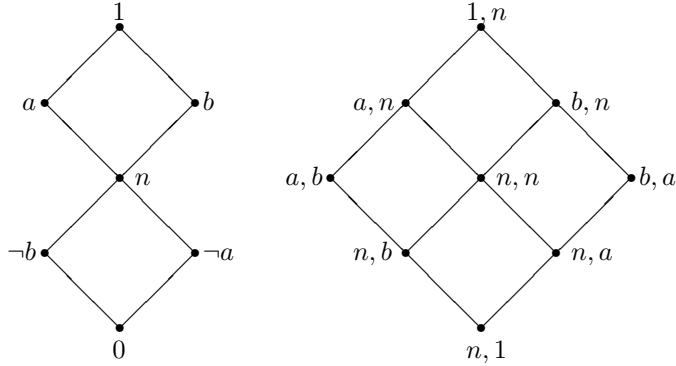
\begin{center}
\setlength{\unitlength}{1mm} 
\begin{picture}(85,50)
\put(0,0){
\begin{picture}(30,50)
\put(15,5){\circle*{1}}

\put(5,15){\circle*{1}}
\put(25,15){\circle*{1}}
\put(15,25){\circle*{1}}
\put(5,35){\circle*{1}}
\put(25,35){\circle*{1}}
\put(15,45){\circle*{1}}

\put(14,1){$0$}
\put(14,46){$1$}
\put(0,14){$ \neg b$}
\put(26,34){$b$}

\put(26,14){$ \neg a$}
\put(2,34){$a$}

\put(17,24){$n$}

\put(15,5){\line(1,1){10}}
\put(15,5){\line(-1,1){10}}
\put(15,45){\line(1,-1){10}}
\put(15,45){\line(-1,-1){10}}

\put(5,15){\line(1,1){20}}
\put(25,15){\line(-1,1){20}}
\end{picture}
}
\put(38,0){
\begin{picture}(50,50)
\put(25,5){\circle*{1}}
\put(15,15){\circle*{1}}
\put(35,15){\circle*{1}}
\put(25,25){\circle*{1}}
\put(15,35){\circle*{1}}
\put(35,35){\circle*{1}}
\put(25,45){\circle*{1}}
\put(5,25){\circle*{1}}
\put(45,25){\circle*{1}}

\put(23,1){$n,1$}
\put(23,46){$1,n$}
\put(8,14){$n,b$}
\put(37,34){$b,n$}

\put(37,14){$n,a$}
\put(8,34){$a,n$}

\put(-1,24){$a,b$}
\put(46,24){$b,a$}

\put(27,24){$n,n$}

\put(25,5){\line(1,1){20}}
\put(25,5){\line(-1,1){20}}
\put(25,45){\line(1,-1){20}}
\put(25,45){\line(-1,-1){20}}

\put(15,15){\line(1,1){20}}
\put(35,15){\line(-1,1){20}}
\end{picture}
}

\end{picture}
\end{center}

\section{Partial valuations\label{parvalpar}}

We introduce the notion of partial valuation on DMF-algebras as a
generalization of the notion of measure of partial probability on a field of
partial sets. As a consequence we adopt the set $T$ partial probability
values, introduced in par. \ref{parprobpar}, as the codomain of partial
valuations. The valuation of a partial set $(A,B)$ with a pair of numbers $%
(x,y)$ was rather natural, but the adoption of such a kind of valuations for
the elements of a DMF-algebra $\mathcal{A}$ needs some explanation. The main
reason is that every $a\in A$, as a consequence of theorem \ref{teocostrdmf}%
, can be seen as a pair constituted by its positive part $a\vee n$ and its
negative part $\lnot a\vee n$. As a result we can transform valued bounded
distributive lattices in valued DMF-algebras with the same construction
which transforms bounded distributive lattices in DMF-algebras

For all $\mathcal{A}\in DMF$, we say that $\overline{v}:A\rightarrow T$ is a 
\textit{partial valuation} on $\mathcal{A}$ if the following axioms are
satisfied:

\begin{enumerate}
\item  $\overline{v}(0)=(0,1),$

\item  $\overline{v}(a\vee b)=\overline{v}(a)+\overline{v}(b)-\overline{v}%
(a\wedge b),$

\item  $\overline{v}(\lnot a))=\sigma (\overline{v}(a)),$

\item  if $n\leq a$ then $(0,0)\preccurlyeq \overline{v}(a)$,
\end{enumerate}

\noindent where $\sigma :T\rightarrow T$ is defined by $\sigma (x,y)=(y,x)$.
If $\mathcal{A}$ is a DMF-algebra and $\overline{v}$ a partial valuation on $%
\mathcal{A}$, we say that $(\mathcal{A},\overline{v})$ is a \textit{valued}
DMF-\textit{algebra.} Obviously every measure of partial probability on a
field of partial sets $\mathcal{G}_{S}$ is a partial valuation on $\mathcal{G%
}_{S}$ and whatever result we can obtain about valued DMF-algebras can be
transferred to partial probability spaces. The following theorem collects
some general properties of partial valuations. In the proof the following
properties of $\sigma $ will be seldom used: i) $\sigma \sigma (x,y)=(x,y)$,
ii) $(x,y)\preceq (w,z)$ iff $\sigma (w,z)\preceq \sigma (x,y)$.

\begin{theorem}
\label{teopropr_v}If $\ \mathcal{A}$ is a DMF-algebra and $\overline{v}$ a
partial valuation on $\mathcal{A}$, then:

\begin{enumerate}
\item  $\overline{v}(1)=(1,0)$,

\item  $\overline{v}(n)=(0,0)$,

\item  if $a\leq n$ then $\overline{v}(a)\preccurlyeq (0,0)$,

\item  $\overline{v}(a)=\overline{v}(a\wedge n)+\overline{v}(a\vee n)$,

\item  $\overline{v}(a)=(\overline{v}(a\vee n)_{0},\overline{v}(\lnot a\vee
n)_{0})$.
\end{enumerate}
\end{theorem}

\begin{proof}
1. In any DM-algebra we have $1=\lnot 0$, so $\overline{v}(\lnot 0)=\sigma (%
\overline{v}((0))=\sigma (0,1)=(1,0)$ by axiom 3.

2. In any DMF-algebra we have both $n\leq n$ and $n\leq \lnot n$, so we have 
$(0,0)\preccurlyeq \overline{v}(n)$ and $(0,0)\preccurlyeq \overline{v}%
(\lnot (n))=\sigma (\overline{v}(n))$ by axiom 4. Then $\sigma \sigma (%
\overline{v}(n))\preccurlyeq \sigma (0,0)$ and $\overline{v}(n)\preccurlyeq
(0,0)$. As $\preccurlyeq $ is a partial order, we conclude with $(0,0)=%
\overline{v}(n)$.

3. If $a\leq n$ then $n\leq \lnot a$ and $(0,0)\preccurlyeq \overline{v}%
(\lnot a)\preccurlyeq \sigma (\overline{v}(a))$, by axiom 4. So $\sigma
\sigma (\overline{v}(a))\preccurlyeq \sigma (0,0)$ and $\overline{v}%
(a)\preccurlyeq (0,0)$.

4. By axiom 2 and point 2 above, we have $\overline{v}(a\vee n)=\overline{v}%
(a)-\overline{v}(a\wedge n)$.

5. We assume $\overline{v}(a)=(x,y)$. By point 4 above, we know that $(x,y)=%
\overline{v}(a\wedge n)+\overline{v}(a\vee n)$. By point 3 above, we have $%
\overline{v}(a\wedge n)\preccurlyeq (0,0)$ and by axiom 4 we have $%
(0,0)\preccurlyeq \overline{v}(a\vee n)$, so there are $z$ and $w$ such that 
$\overline{v}(a\wedge n)=(0,z)$ and $\overline{v}(a\vee n)=(w,0)$. Thus $%
(x,y)=(0,z)+(w,0)$ and $x=w$ and $y=z$. We can conclude that $\overline{v}%
(a)_{0}=\overline{v}(a\vee n)_{0}$ and $\overline{v}(a)_{1}=\overline{v}%
(a\wedge n)_{1}$. But $\sigma (\overline{v}(a\wedge n))=\overline{v}(\lnot
(a\wedge n))=\overline{v}(\lnot a\vee n)$, so $\overline{v}(a)_{1}=\overline{%
v}(\lnot a\vee n)_{0}$.
\end{proof}

\begin{theorem}
\label{teo333}If $\mathcal{A}$ and $\mathcal{B}$ are DMF-algebras, $\varphi :%
\mathcal{A}\rightarrow \mathcal{B}$ is a morphism and $\overline{v}$ is a
partial valuation on $\mathcal{B}$, then $\theta =\overline{v}\circ \varphi $
is a partial valuation on $\mathcal{A}$.
\end{theorem}

\begin{proof}
We show that $\theta $ satisfies the four axioms of partial valuation.

1. $\theta (0^{\mathcal{A}})=\overline{v}(\varphi (0^{\mathcal{A}}))=%
\overline{v}(0^{\mathcal{B}})=(0,1)$.

2. 
\begin{eqnarray*}
\theta (a\vee b)&=&\overline{v}(\varphi (a)\vee \varphi (b)) \\
&=&\overline{v}(\varphi (a))+\overline{v}(\varphi (b))-\overline{v}(\varphi
(a)\wedge \varphi (b)) \\
&=&\overline{v}(\varphi (a))+\overline{v}(\varphi (b))-\overline{v}(\varphi
(a\wedge b)) \\
&=&\theta (a)+\theta (b)-\theta (a\wedge b).
\end{eqnarray*}

3. $\theta (\lnot a)=\overline{v}(\varphi (\lnot a))=$ $\overline{v}(\lnot
\varphi (a))=\sigma (\overline{v}(\varphi (a)))=\sigma (\theta (a))$.

4. Let $n^{\mathcal{A}}\leq a$, then $n^{\mathcal{B}}=\varphi (n^{\mathcal{A}%
})\leq \varphi (a)$, so $(0,0)\preccurlyeq \overline{v}(\varphi (a))=\theta
(a)$.
\end{proof}

The following theorem shows that the general construction of DMF-algebras
from bounded distributive lattices can be transferred to valued algebras.

\begin{theorem}
\label{teom1}If $\mathcal{A}$ is a bounded distributive lattice and $v $ a
valuation on $\mathcal{A}$, then $\overline{v}_{v }:\pi (\mathcal{A}%
)\rightarrow T$ defined by $\overline{v}_{v }(x,y)=(v (x),v (y))$ is a
partial valuation on $\pi (\mathcal{A})$.
\end{theorem}

\begin{proof}
Firstly we observe that $\pi (\mathcal{A})$ is a DMF-algebra, by theorem \ref
{teopi}, then we verify that $\overline{v}_{v }$ effectively takes $\pi (A)
$ to $T$. It is enough to show that, for all $(a,b)$ in $\pi (A)$, $v
(a)+v (b)$ belongs to $[0,1]$. \ As $v $ is a valuation on $\mathcal{A}$%
, 
\begin{equation*}
v (a)+v (b)=v (a\vee b)+v (a\wedge b)=v (a\vee b),
\end{equation*}
where $a\wedge b=0$ (because $(a,b)$ is in $\pi (A)$) and obviously $v
(a\vee b)\in \lbrack 0,1]$. We have only to show that $\overline{v}_{v }$
satisfies the four axioms of partial valuations.

Axiom 1: $\overline{v}_{v }(0)=(v (0),v (1))=(0,1)$, because $v $ is
a bounded lattices valuation.

Axiom 2: we have 
\begin{equation*}
\overline{v}_{v }((a,b)\vee (a^{\prime },b^{\prime }))=\overline{v}_{v
}((a\vee a^{\prime },b\wedge b^{\prime }))
=(v (a\vee a^{\prime }),v (b\wedge b^{\prime }))
\end{equation*}
and 
\begin{equation*}
\overline{v}_{v }((a,b)\wedge (a^{\prime },b^{\prime }))=(v (a\wedge
a^{\prime }),v (b\vee b^{\prime })).
\end{equation*}
As $v $ is a lattices valuation, we have both 
\begin{equation*}
v (a)+v (a^{\prime })=v (a\vee a^{\prime })+v (a\wedge a^{\prime })
\text{ and } 
v (b)+v (b^{\prime })=v (b\vee b^{\prime })+v (b\wedge b^{\prime }),
\end{equation*}
we can conclude that 
\begin{eqnarray*}
\overline{v}_{v }((a,b)\vee (a^{\prime },b^{\prime }))+\overline{v}_{v
}((a,b)\wedge (a^{\prime },b^{\prime }))&=&(v (a)+v (a^{\prime }),v
(b)+v (b^{\prime })) \\
&=&(v (a),v (b))+(v (a^{\prime }),v (b^{\prime })) \\
&=&\overline{v}_{v }(a,b)+\overline{v}_{v }(a^{\prime },b^{\prime }).
\end{eqnarray*}

Axiom 3: $\overline{v}_{v }(\lnot (a,b))=(v (b),v (a))=\sigma (v
(a),v (b))=\sigma (\overline{v}_{v }(a,b))$.

Axiom 4. We assume $n\leq (a,b)$. In $\pi (\mathcal{A})$ we have $n=(0^{%
\mathcal{A}},0^{\mathcal{A}})$, so $b=0^{\mathcal{A}}$ and $v (b)=0$,
because $v $ is a lattice valuation on $\mathcal{A}$. Thus $%
(0,0)\preccurlyeq (v (a),0)=(v (a),v (b))=\overline{v}_{v }(a,b)$.
\end{proof}

As a consequence we can define, for any valuation $v $ on a bounded
distributive lattice $\mathcal{A}$, the \textit{partial valuation} $%
\overline{v}_{v }$ \textit{on} $\pi (\mathcal{A})$ \textit{induced by}$\
v $ as the function $\overline{v}_{v }(x,y)=(v (x),v (y))$. The
following theorems show that every partial valuation on a DMF-algebra can be
obtained in this way.

\begin{theorem}
\label{teom2}For every partial valuation $\overline{v}$ on a DMF-algebra $%
\mathcal{B}$, the function $v :$ $\nabla _{\mathcal{B}}\rightarrow \lbrack
0,1]$ defined setting $v (x)=\overline{v}(x)_{0}$ is a valuation on the
bounded lattice $\nabla _{\mathcal{B}}$.
\end{theorem}

\begin{proof}
By point 1) of theorem \ref{teopropr_v} we have $\overline{v}(1^{\mathcal{B}%
})=(1,0)$, thus $v (1^{\nabla _{\mathcal{B}}})=1$ because $1^{\mathcal{B}%
}=1^{\nabla _{\mathcal{B}}}$. By point 2 of theorem \ref{teopropr_v} we have$%
\overline{v}(n^{\mathcal{B}})=(0,0)$, thus $v (0^{\nabla _{\mathcal{B}}})=0
$ because $n^{\mathcal{B}}=0^{\nabla _{\mathcal{B}}}$. Finally $v $
satisfies (\ref{1}) of par. \ref{parboolval} because $\overline{v}$ satisfies
axiom 2, so $v $ is a bounded lattices valuation.
\end{proof}

\begin{theorem}
\label{teom3}If $\overline{v}$ is a partial valuation on a DMF-algebra $%
\mathcal{B}$ then there is a bounded distributive lattice $\mathcal{A}$ and
a valuation $v $ on $\mathcal{A}$ such that:

\begin{enumerate}
\item  there is a monomorphism $\varphi :\mathcal{B}\rightarrow \pi (%
\mathcal{A})$ ,

\item  for all $b\in B$, $\overline{v}(b)=\overline{v}_{v }(\varphi (b))$,
where $\overline{v}_{v }$ is the partial valuation induced by $v $.
\end{enumerate}
\end{theorem}

\begin{proof}
1. If we set $\mathcal{A}=\nabla _{\mathcal{B}}$, then by theorem \ref
{teocostrdmf} there is a monomorphism $\varphi :\mathcal{B}\rightarrow \pi
(\nabla _{\mathcal{B}})$ defined setting $\varphi (b)=(b\vee n,\lnot b\vee n)
$.

2. From the partial valuation $\overline{v}$ on $\mathcal{B}$ we get, by the
above theorem, a valuation $v $ on $\nabla _{\mathcal{B}}$, setting $v
(x)=\overline{v}(x)_{0}$. Thus, by theorem \ref{teom1}, we have a partial
valuation $\overline{v}_{v }$ on $\pi (\nabla _{\mathcal{B}})$ setting $%
\overline{v}_{v }(x,y)=(v (x),v (y))$. We show that, for all $b\in B$, 
$\overline{v}(b)=\overline{v}_{v }(\varphi (b))$, what amounts to say 
\begin{equation*}
\overline{v}(b)=(v (b\vee n),v (\lnot b\vee n))=(\overline{v}(b\vee
n)_{0},\overline{v}(\lnot b\vee n)_{0}).
\end{equation*}
As $(\lnot b\vee n)_{0}=(\lnot (b\wedge n))_{0}=(b\wedge n)_{1}$, we can
reduce ourselves to prove 
\begin{equation*}
\overline{v}(b)=(\overline{v}(b\vee n)_{0},\overline{v}(b\wedge n)_{1}).
\end{equation*}
In fact, 
\begin{eqnarray*}
\overline{v}(b)&=&\overline{v}(b\vee n)+\overline{v}(b\wedge n) \\
&=&
(\overline{v}(b\vee n)_{0},\overline{v}(b\vee n)_{1})+
(\overline{v}(b\wedge n)_{0},\overline{v}(b\wedge n)_{1}) \\
&=&(\overline{v}(b\vee n)_{0},0)+(0,\overline{v}(b\wedge n)_{1}) \\
&=&(\overline{v}(b\vee n)_{0},\overline{v}(b\wedge n)_{1}),
\end{eqnarray*}
where the first line follows from point 4 of theorem \ref{teopropr_v} and the
third line from axiom 4, because $n\leq b\vee n$ implies $(0,0)\preccurlyeq 
\overline{v}(b\vee n)$, i.e. $\overline{v}(b\vee n)_{1}=0$, and in the same
way $\overline{v}(b\wedge n)_{0}=0$.
\end{proof}

\begin{theorem}
Every partial probability space $(S,\mathcal{D}(S),\overline{v})$ is the
partial probability space associated to a classical probability space $(S,%
\mathcal{P}(S),p)$.
\end{theorem}

\begin{proof}
We define $p:\mathcal{P}(S)\rightarrow \lbrack 0,1]$ as follows. By theorem 
\ref{teom2}, we know that $v:\nabla _{\mathcal{D}(S)}\rightarrow \lbrack 0,1]
$ defined by $v(x)=\overline{v}(x)_{0}$ is a valuation on $\nabla _{\mathcal{%
D}(S)}$. We know also that $\psi :\mathcal{P}(S)\rightarrow \nabla _{%
\mathcal{D}(S)}$ defined by $\psi (A)=(A,\emptyset )$, for all $A\subseteq S$%
, is a bounded lattices isomorphism. So, by theorem \ref{teo333} , $v\circ
\psi $ is a valuation on $\mathcal{P}(S)$. So we define $p=$ $v\circ \psi $,
i.e. $p(A)=\overline{v}(A,\emptyset )_{0}$, for all $A\subseteq S$. To show
that $(S,\mathcal{D}(S),\overline{v})$ is the partial probability space
associated to the classical probability space $(S,\mathcal{P}(S),p)$, we
have to prove that $\overline{v}(A,B)=(p(a),p(B))$, for every partial event $%
(A,B)\in D(S)$: in fact $\overline{v}(A,B)=(\overline{v}(A,\emptyset )_{0},%
\overline{v}(A,\emptyset )_{1})=(p(A),p(B))$ by point 5 of theorem \ref
{teopropr_v} and by definition of $p$.
\end{proof}

Given a valued DMF-algebra $(\mathcal{A},\overline{v})$, we define a
function $u:A\rightarrow \lbrack 0,1]$ setting $u(a)=1-(\overline{v}(a)_{0}+%
\overline{v}(a)_{1})$. The number $u(a)$ is the \textit{degree of
indetermination} of $a$. The following theorem shows that the values of
elements with the same degree of indetermination are linearly ordered.

\begin{theorem}
If $u(a)=u(b)$ then $\overline{v}(a)\preccurlyeq \overline{v}(b)$ or $%
\overline{v}(b)\preccurlyeq \overline{v}(a)$.
\end{theorem}

\begin{proof}
If $u(a)=u(b)$ then $\overline{v}(a)_{0}+\overline{v}(a)_{1}=\overline{v}%
(b)_{0}+\overline{v}(b)_{1}$. In general, if $x+y=k=x^{\prime }+y^{\prime }$%
, then we distinguish two cases: i) $x\leq x^{\prime }$ implies $k-x^{\prime
}\leq k-x$ and so $y^{\prime }\leq y$; ii) $x^{\prime }\leq x$ implies $%
y\leq y^{\prime }$. Thus $\overline{v}(a)_{0}\leq \overline{v}(b)_{0}$ and $%
\overline{v}(b)_{1}\leq \overline{v}(a)_{1},$ or $\overline{v}(b)_{0}\leq 
\overline{v}(a)_{0}$ and $\overline{v}(a)_{1}\leq \overline{v}(b)_{1}$. i.e $%
\overline{v}(a)\preccurlyeq \overline{v}(b)$ or $\overline{v}(b)\preccurlyeq 
\overline{v}(a)$.
\end{proof}

\begin{theorem}
The values of boolean elements are linearly ordered
\end{theorem}

\begin{proof}
We show that $u(a)=0$ for all $a\in K$.
In fact, if $a\in K$ then $a\vee \lnot a = 1$, by point 2) of theorem \ref{teoBA},
and so $\overline{v}(a\vee \lnot a)=(1,0)$, by point 1) of theorem \ref{teopropr_v}, and
$\overline{v}(a\wedge \lnot a)=(0,1)$, by axiom 3. Thus 
\begin{equation*}
1=\overline{v}(a\vee \lnot a)_{0}=\overline{v}(a)_{0}+\overline{v}(\lnot
a)_{0}-\overline{v}(a\wedge \lnot a)_{0}=\overline{v}(a)_{0}+\overline{v}%
(\lnot a)_{0}=\overline{v}(a)_{0}+\overline{v}(a)_{1}
\end{equation*}
and $u(a)=0$. So $K$ is linearly ordered by the theorem above.
\end{proof}

\section{Partial isotone valuations}

We say that a partial valuation $\overline{v}$ on a DMF-algebra $\mathcal{A}$
is isotone if $a\leq b$ implies $\overline{v}(a)\preceq \overline{v}(b)$,
for all $a$, $b\in A$. We show that if $\nabla $ of $\mathcal{A}$
is a Boolean algebra, then every partial valuation on $\mathcal{A}$ is
isotone.

\begin{lemma}
If $\mathcal{A}$ is a DMF-algebra and $\nabla $ of $\mathcal{A}$ is
a Boolean algebra, then every partial valuation on $\mathcal{A}$ is isotone
on $\nabla $.
\end{lemma}

\begin{proof}
We suppose that $x$, $y\in \nabla $ and $x\leq y$. We denote with $x^{\circ }
$ the complement of $x$ in the Boolean algebra $\nabla $ and set $z=y\wedge
x^{\ast }$. Then 
\begin{equation*}
x\vee z=(x\vee y)\wedge (x\vee x^{\ast })=(x\vee y)\wedge 1=x\vee y=y
\end{equation*}
and 
\begin{equation*}
x\wedge z=y\wedge (x\wedge x^{\ast })=y\wedge n=n.
\end{equation*}
Thus 
\begin{equation*}
\overline{v}(y)=\overline{v}(x\vee z)
=\overline{v}(x)+\overline{v}(z)-\overline{v}(x\wedge z)
=\overline{v}(x)+\overline{v}(z)-\overline{v}(n)
=\overline{v}(x)+\overline{v}(z).
\end{equation*}
As $x$, $z\in \nabla $, there are $q$, $t\in \lbrack 0,1]$ such that $%
\overline{v}(x)=(q,0)$ and $\overline{v}(z)=(t,0)$, so $\overline{v}%
(y)=(q+t,0)$ and $\overline{v}(x)\preccurlyeq \overline{v}(y)$.
\end{proof}

\begin{lemma}
If $\mathcal{A}$ is a DMF-algebra, $\overline{v}$ is a partial valuation on $%
\mathcal{A}$ and $\overline{v}$ is isotone on $\nabla $, then $\overline{v}$
is isotone on $\nabla \cup \Delta $.
\end{lemma}

\begin{proof}
We suppose that $x\leq y$ and distinguish four cases.

\begin{enumerate}
\item  If $x$, $y\in \nabla $ then $\overline{v}(x)\preccurlyeq \overline{v}%
(y)$ by hypothesis.

\item  If $x$, $y\in \Delta $ then $\lnot x$, $\lnot y\in \nabla $ and $%
\lnot y\leq \lnot x$ so $\overline{v}(\lnot y)\preccurlyeq \overline{v}%
(\lnot x)$ by the preceding point, thus $\sigma (\overline{v}%
(y))\preccurlyeq \sigma (\overline{v}(x))$ and $\overline{v}(x)\preccurlyeq 
\overline{v}(y)$.

\item  If $x\in \Delta $ and $y\in \nabla $ then $x\leq n\leq y$. By axiom 4
and point 3 of theorem \ref{teopropr_v}, we have $\overline{v}%
(x)\preccurlyeq (0,0)\preccurlyeq \overline{v}(y)$.

\item  If $y\in \Delta $ and $x\in \nabla $ then $y\leq n\leq x$. By
hypothesis $y\leq x$, so $x=y$ and $\overline{v}(x)=\overline{v}(y)$.
\end{enumerate}
\end{proof}

\begin{lemma}
If $\mathcal{A}$ is a DMF-algebra, $\overline{v}$ is a partial valuation on $%
\mathcal{A}$ and $\overline{v}$ is isotone on $\nabla $, then $\overline{v}$
is isotone on $\mathcal{A}$.
\end{lemma}

\begin{proof}
If $x\leq y$, then $x\vee n\leq y\vee n$ and $x\wedge n\leq y\wedge n$. As $%
\overline{v}$ is isotone on $\nabla $, it is also isotone on $\nabla \cup
\Delta $, by the preceding lemma, so $\overline{v}(x\vee n)\preccurlyeq 
\overline{v}(y\vee n)$ and $\overline{v}(x\wedge n)\preccurlyeq \overline{v}%
(y\wedge n)$. As $+$ is isotone on $T$ with respect to \ $\preccurlyeq $, 
\begin{equation*}
\overline{v}(x\vee n)+\overline{v}(x\wedge n)\preccurlyeq \overline{v}(y\vee
n)+\overline{v}(y\wedge n),
\end{equation*}
thus $\overline{v}(x)\preccurlyeq \overline{v}(y)$ by point 4 of theorem \ref
{teopropr_v}
\end{proof}

\begin{theorem}
\label{teovalisot}If $\mathcal{A}$ is a DMF-algebra and $\nabla $
of $\mathcal{A}$ is a Boolean algebra, then every partial valuation on $%
\mathcal{A}$ is isotone on the whole $\mathcal{A}$.
\end{theorem}

\begin{proof}
The theorem follows from the above lemmas.
\end{proof}

As a consequence of this theorem, every measure of partial probability $\mu $
on a partial field of set $\mathcal{G}_{S}$ on $S$ is isotone, when $\nabla
_{\mathcal{G}_{S}}$ is a Boolean algebra: in particular, $\mu $ is isotone
when $\mathcal{G}_{S}=\mathcal{D}(S)$, because $\nabla _{\mathcal{D}(S)}=%
\mathcal{P}(S)$.

\section{Relativized partial valuations\label{parrpv}}

In every bounded lattice $\mathcal{A}$, we can associate to every $a\in A$
the relativization $f_{a}:\mathcal{A}\rightarrow \lbrack 0,a]$ defined by $%
f_{a}(x)=a\wedge x$. If $\mathcal{A}$ is a Boolean algebra, the
relativization is a Boolean omomorphism. As we have seen in theorem \ref
{teorelativizedval}, we can associate to every valued Boolean algebra $(%
\mathcal{A},v)$ and to every $a\in A$ such that $v(a)\neq 0$, a relativized
valuation $v_{a}$ on $[0,a]$.

If we want to do the same thing in the context\ of partial valuations on
DMF-algebras, we are compelled to give a new definition of relativization
because $f_{a}$ is no more a morphism. Firstly we observe that, in every
distributive lattice $\mathcal{A}$, if $a\leq b$ we can define a morphism $%
f_{b}^{a}:A\rightarrow \lbrack a,b]$ by 
\begin{equation*}
f_{b}^{a}(x)=(x\vee a)\wedge b.
\end{equation*}
(An equivalent definition is $f_{b}^{a}(x)=(x\wedge b)\vee a$, because by
modularity $a\leq b$ implies $(x\vee a)\wedge b=(x\wedge b)\vee a$.) It can
be easily seen that $f_{b}^{a}$ preserves $\wedge $ and $\vee $. If $%
\mathcal{A}$ is bounded then $f_{b}^{a}$ preserves also $0$ and $1$. We call 
$f_{b}^{a}$ the \textit{relativization associated to} $[a,b]$. The following
theorem gives a sufficient condition for the existence of a relativization
in a DMF-algebra.

\begin{theorem}
\label{teorelpar}If $\mathcal{A}$ is a DMF-algebra and $\lnot a\leq a$, then 
$[\lnot a,a]$ can be expanded to a DMF-algebra $\mathcal{B}$ and the
relativization $f_{a}^{\lnot a}$associated to $[\lnot a,a]$ is a morphism of
DMF-algebras from $\mathcal{A}$ to $\mathcal{B}$.
\end{theorem}

\begin{proof}
We define a structure $\mathcal{B}$ of domain $[\lnot a,a]$ setting $\wedge ^{%
\mathcal{B}}=\wedge $, $\vee ^{\mathcal{B}}=\vee $, $\lnot ^{\mathcal{B}%
}=\lnot $, $0^{\mathcal{B}}=a$, $1^{\mathcal{B}}=\lnot a$ and $n^{\mathcal{B}%
}=n$. We observe only that $[\lnot a,a]$ is closed with respect to $\lnot $,
because $\lnot a\leq x\leq a$ implies $\lnot a\leq \lnot x\leq a$, and $n\in 
$ $[\lnot a,a]$ because $\lnot a=a\wedge \lnot a\leq n\leq a\vee \lnot a=a$,
by the normality axiom. As DMF-algebra axioms are equational, they are
inherited by $\mathcal{B}$.

It is immediate to verify that $f_{a}^{\lnot a}$ preserves $\wedge $, $\vee $%
, $0$ and $1$. $f_{a}^{\lnot a}$ preserves $\lnot $: 
\begin{equation*}
\lnot (f_{a}^{\lnot a}(x))=\lnot ((x\vee \lnot a)\wedge a) 
=(\lnot x\wedge a)\vee \lnot a 
=(\lnot x\vee \lnot a)\wedge a 
=f_{a}^{\lnot a}(\lnot x).
\end{equation*}
Finally $f_{a}^{\lnot a}$ preserves $n$, because $\lnot a\leq n\leq a$
implies $f_{a}^{\lnot a}(n)=(n\vee \lnot a)\wedge a=n$.
\end{proof}

Now we can introduce relativized partial valuations on DMF-algebras. Let $%
\mathcal{A}$ be a DMF-algebra and let $h\in \nabla$.
We observe that $h\in \nabla $ is equivalent to $\lnot h\leq h$.
Let $\overline{v}$ be an isotone partial valuation on $\mathcal{A}$ 
such that $\overline{v}(h)_{0}\neq 0$. 
We define $\overline{v}_{h}:[\lnot h,h]\rightarrow R^{2}$ setting 
\begin{equation*}
\overline{v}_{h}(x)=\overline{v}(x)\cdot \frac{1}{\overline{v}(h)_{0}}.
\end{equation*}
We call $\overline{v}_{h}$ the \textit{relativized partial valuation}
associated to $\overline{v}$ and $h$. The following theorem shows that $%
\overline{v}_{h}$ is a partial valuation on $[\lnot h,h]$, indeed.

\begin{lemma}
If $\overline{v}$ is an isotone partial valuation on $\mathcal{A}$ and $h\in
\nabla$, then $\lnot h\leq x\leq h$ implies $\overline{v%
}(x)_{0}+\overline{v}(x)_{1}\leq \overline{v}(h)_{0}$.
\end{lemma}

\begin{proof}
From $\lnot h\leq x\leq h$ we have $\lnot h\leq \lnot x\leq h$ and then $%
x\vee \lnot x\leq h$. By the normality axiom we have both $(x\vee n)\vee
(\lnot x\vee n)=x\vee \lnot x$ and $(x\vee n)\wedge (\lnot x\vee n)=n$. As $%
\overline{v}$ is isotone, 
\begin{eqnarray*}
\overline{v}(h)\succcurlyeq \overline{v}(x\vee \lnot x) 
=\overline{v}((x\vee n)\vee (\lnot x\vee n)) 
&=&\overline{v}(x\vee n)+\overline{v}(\lnot x\vee n)-\overline{v}(n) \\
&=&\overline{v}(x\vee n)+\overline{v}(\lnot x\vee n).
\end{eqnarray*}
By point 5 of theorem \ref{teopropr_v} we have $\overline{v}(x)=(\overline{v}%
(x\vee n)_{0},\overline{v}(\lnot x\vee n)_{0})$ and so 
\begin{equation*}
\overline{v}(x)_{0}+\overline{v}(x)_{1}=\overline{v}(x\vee n)_{0}+\overline{v%
}(\lnot x\vee n)_{0}
\end{equation*}
and finally $\overline{v}(h)_{0}\geq \overline{v}(x\vee n)_{0}+\overline{v}%
(\lnot x\vee n)_{0}=\overline{v}(x)_{0}+\overline{v}(x)_{1}$.
\end{proof}

\begin{theorem}
\label{teovalparrel}If $\overline{v}$ is an isotone partial valuation on 
$\mathcal{A}$, $h\in \nabla$ and $\overline{v}%
(h)_{0}\neq 0$, then $\overline{v}_{h}$ is a function from $[\lnot h,h]$ to $%
T$ that satisfies the axioms of partial valuation.
\end{theorem}

\begin{proof}
Firstly we show that $\overline{v}_{h}$ takes values in $T$. We suppose $%
h\leq x\leq \lnot h$ and $\overline{v}(x)=(q,t)$. Then $\overline{v}%
_{h}(x)=(q,t)\cdot 1/\overline{v}(h)_{0}$. We have to show that:

\begin{description}
	\item[i)] $0\leq 
	\frac{q}{\overline{v}(h)_{0}},\frac{t}{\overline{v}(h)_{0}} \leq 1,$
	\item[ii)] $0\leq \frac{q+t}{\overline{v}(h)_{0}}\leq 1$.
\end{description}

i). By hypothesis $0<\overline{v}(h)_{0}$ and $0\leq q,t$, so we have both $%
0\leq \frac{q}{\overline{v}(h)_{0}}$ and $0\leq \frac{t}{\overline{v}(h)_{0}}
$. As $\overline{v}$ is isotone, $x\leq h$ we have $\overline{v}%
(x)\preccurlyeq \overline{v}(h)$ and $q\leq \overline{v}(h)_{0}$. As $\lnot
h\leq x$, we have $\lnot x\leq h$, so $\overline{v}(\lnot x)\preccurlyeq 
\overline{v}(h)$ and then $(t,q)=\sigma (\overline{v}(x))\preccurlyeq m(h)$
and finally $t\leq \overline{v}(h)_{0}$. In conclusion, $q/\overline{v}%
(h)_{0}$, $t/\overline{v}(h)_{0}\leq 1$.

ii). By the above lemma, $q+t\leq \overline{v}(h)_{0}$ and so $(q+t)/\overline{%
v}(h)_{0}\leq 1$. On the other side, from $0\leq q$, $t$ we have $0\leq q+t$
and $0\leq (q+t)/\overline{v}(h)_{0}$, because $\overline{v}(h)_{0}$ is
positive.

We expand $[\lnot h,h]$ to a DMF-algebra $\mathcal{B}$ as in theorem \ref{teorelpar}
and prove that $\overline{v}_{h}$ is a partial valuation on $\mathcal{B}$.

Axiom 1. 
\begin{equation*}
\overline{v}_{h}(0^{\mathcal{B}})=\overline{v}_{h}(\lnot h) 
=\overline{v}(\lnot h)\cdot \frac{1}{\overline{v}(h)_{0}} 
=\sigma (\overline{v}(h))\cdot \frac{1}{\overline{v}(h)_{0}} 
=(\frac{\overline{v}(h)_{1}}{\overline{v}(h)_{0}},\frac{\overline{v}(h)_{0}}{%
\overline{v}(h)_{0}}) 
=(0,1),
\end{equation*}
because $0<\overline{v}(h)_{0}$ holds by hypothesis and $\overline{v}%
(h)_{1}=0$ follows from $h\in \nabla $ and, consequently, $(0,0)\preccurlyeq 
\overline{v}(h)$.

Axiom 2. 
\begin{equation*}
\overline{v}_{h}(x\vee y)=\overline{v}(x\vee y)\cdot \frac{1}{\overline{v}%
(h)_{0}} 
=\frac{\overline{v}(x)}{\overline{v}(h)_{0}}+\frac{\overline{v}(y)}{%
\overline{v}(h)_{0}}-\frac{\overline{v}(x\wedge y)}{\overline{v}(h)_{0}} 
=\overline{v}_{h}(x)+\overline{v}_{h}(y)+\overline{v}_{h}(x\wedge y).
\end{equation*}

Axiom 3. 
\begin{eqnarray*}
\overline{v}_{h}(\lnot x)&=&\overline{v}(\lnot x)\cdot \frac{1}{\overline{v}%
(h)_{0}}
=\sigma (\overline{v}(x))\cdot \frac{1}{\overline{v}(h)_{0}} 
=(\frac{\overline{v}(x)_{1}}{\overline{v}(h)_{0}},\frac{\overline{v}(x)_{0}}{%
\overline{v}(h)_{0}}) \\
&=&\sigma (\frac{\overline{v}(x)_{0}}{\overline{v}(h)_{0}},\frac{\overline{v}%
(x)_{1}}{\overline{v}(h)_{0}}) 
=\sigma (\overline{v}(x))\cdot \frac{1}{\overline{v}(h)_{0}}) 
=\sigma (\overline{v}_{h}(x)).
\end{eqnarray*}

Axiom 4. If $n\leq x$ then $(0,0)\preccurlyeq \overline{v}(x)$ and so $0\leq 
\overline{v}(x)_{0}$ and $0=\overline{v}(x)_{1}$, thus $0\leq \overline{v}%
(x)_{0}/\overline{v}(h)_{0}$ and $0=\overline{v}(x)_{1}/\overline{v}(h)_{0}$%
. Thus 
\begin{equation*}
(0,0)\preccurlyeq (\frac{\overline{v}(x)_{0}}{\overline{v}(h)_{0}},\frac{%
\overline{v}(x)_{1}}{\overline{v}(h)_{0}})=\overline{v}_{h}(x).
\end{equation*}
\end{proof}

\hfill

The above concepts can be interpreted in a probability context, defining the
partial valuation $\overline{v}$ on the DMF-algebra $\mathcal{D}(S)$. If we
choose an event $(H,H^{\prime })$ in $D(S)$ as a condition then, by the theorem above,
we can speak of the relativized partial valuation
$\overline{v}_{H,H^{\prime }}$
only if 
i) $\overline{v}$ is isotone,
ii) $(H,H^{\prime })\in \nabla$ and
iii) $\overline{v}(H,H^{\prime })_{0}>0$.
Condition i) is always satisfied by theorem \ref{teovalisot}. Condition ii) 
implies $H^{\prime }=\emptyset$. When these conditions are satisfied, 
we can define
a relativized partial valuation $\overline{v}_{H,\emptyset }$:$%
[(\emptyset ,H),(H,\emptyset )]\rightarrow T$ setting 
\begin{equation*}
\overline{v}_{H,\emptyset }(X,Y)=\overline{v}(X,Y)\cdot \frac{1}{\overline{v}%
(H,H^{\prime })_{0}}.
\end{equation*}
As in classical probability the assumption of a condition $H\subseteq S$
causes the passage from $\mathcal{P}(S)$ to the algebra of relativized
events $\mathcal{P}(H)$, so in partial probability theory the assumption of
$(H,\emptyset)$ causes the passage from $\mathcal{D}(S)$ to $\mathcal{D}(H)$%
, the algebra of all partial sets like $(X\cap H,Y\cap H)$, as $(X,Y)$
varies in $\mathcal{D}(S)$. As a consequence
of theorem \ref{teovalparrel}, we know that $\overline{v}_{H,\emptyset }$
satisfies the axioms of partial valuations introduced in par. \ref{parvalpar}. 

\section{Conditional partial valuations}

We remark that the existence of the relativized partial valuation $\overline{%
v}_{h}$ associated to $h$ depends on three conditions: i) $h$ must satisfy  
$h\in \nabla $ or, equivalently, $\lnot h\leq h$,
ii) $\overline{v}$ must be
isotone, iii) $h$ must satisfy $\overline{v}(h)_{0}\neq 0$. When all these
conditions are satisfied, we can define the \textit{conditional partial
valuation associated to} $v$ \textit{and} $h$ setting $\overline{v}(x|h)=%
\overline{v}_{h}(f_{h}^{\lnot h}(x))$. As $f_{h}^{\lnot h}$ is a morphism
from $\mathcal{A}$ to $[\lnot h,h]$ and $\overline{v}_{h}$ is a partial
valuation on $[\lnot h,h]$, the function $\overline{v}(-|h)$ is a partial
valuation on $\mathcal{A}$, by theorem \ref{teo333}.

If we confine ourselves to the domain of partial sets, we have
the conditional partial valuation
$\overline{v}(-|H,\emptyset):\mathcal{D}(S)\rightarrow T$ setting 
\begin{equation*}
\overline{v}(X,Y|H,\emptyset)=\overline{v}_{H,\emptyset }(f_{H,\emptyset
}^{\emptyset ,H}(X,Y)).
\end{equation*}
We call $\overline{v}(X,Y|H,\emptyset)$ the \textit{conditional partial
probability of} $(X,Y)$ \textit{with respect to} $(H,\emptyset)$. As a
consequence of theorem \ref{teo333}, $\overline{v}(-|H,\emptyset )=%
\overline{v}_{H,\emptyset }\circ f_{H,\emptyset }^{\emptyset ,H}$ is a partial
valuation on $\mathcal{D}(S)$ and a partial probability measure for partial
events.

Now we show two results about conditional partial valuations. The first is a weak form
of Bayes's Theorem. We can give a lattice-theoretic formulation of
Bayes's Theorem as follows.
Let $\mathcal{A}$ be a bounded lattice and $v$ a valuation on $\mathcal{A}$.
If $e$, $h\in A$ are such that $v(e)$, $v(h)\neq 0$, then we can consider
the conditional valuations $v(-|e)$ and $v(-|h)$ associated to $e$ and $h$.
Bayes's Theorem is the following fundamental relation between $v(-|e)$ and $%
v(-|h)$: 
\begin{equation*}
v(h|e)=v(e|h)\frac{v(h)}{v(e)}.
\end{equation*}
The proof is straightforward, because $v(h|e)v(e)=v(h\wedge e)=v(e\wedge
h)=v(e|h)v(h)$.

From an algebraic point of view, the role of the two conditional valuations $%
v(-|e)$ and $v(-|h)$ is perfectly symmetrical. The situation is slightly
different in probability theory, where $h$ and $e$ are to be understood
respectively as an hypothesis to be tested and an experimental evidence,
thus $v(h|e)$ is the `posterior probability' of the hypothesis $h$ and $%
v(e|h)$ is the `likelihood' of the hypothesis (the probability of the
hypothesis conditional on the data). $v(h)$ and $v(e)$ are respectively the
`prior probability' of the hypothesis and the probability of the data. 

A weak form of Bayes's Theorem is available for conditional
partial valuations, an then for partial probability.

\begin{theorem}
If $\overline{v}$ is isotone and $e$, $h\in \nabla $, with $\overline{v}%
(e)_{0}\neq 0$ and $\overline{v}(h)_{0}\neq 0$, then 
\begin{equation*}
\overline{v}(h|e)=\overline{v}(e|h)\frac{v(h)_{0}}{v(e)_{0}}.
\end{equation*}
\end{theorem}

\begin{proof}
As we have seen in par. \ref{parrpv}, we can speak of the conditional
partial valuation $\overline{v}(x|h)=\overline{v}_{h}(f_{h}^{\lnot h}(x))$
when the following three conditions are satisfied: $h\in \nabla $, ii) $%
\overline{v}$ is isotone, iii) $\overline{v}(h)_{0}\neq 0$, what is
guaranteed by our hypothesis. The same holds for $\overline{v}(x|e)=%
\overline{v}_{e}(f_{e}^{\lnot e}(x))$. So we have 
\begin{eqnarray*}
\overline{v}(h|e) &=&\overline{v}_{e}(f_{e}^{\lnot e}(h)) \\
&=&\overline{v}(f_{e}^{\lnot e}(h))\cdot \frac{1}{\overline{v}(e)_{0}} \\
&=&\overline{v}(f_{h}^{\lnot h}(e))\cdot \frac{1}{\overline{v}(e)_{0}} \\
&=&\overline{v}(e|h)\cdot \frac{\overline{v}(h)_{0}}{\overline{v}(e)_{0}}.
\end{eqnarray*}
The second line follows by definition of $\overline{v}_{e}$ in par. \ref
{parrpv}. The third line is justified because $e$, $h\in \nabla $ implies $%
f_{h}^{\lnot h}(e)=f_{e}^{\lnot e}(h)$: in fact, $f_{h}^{\lnot
h}(e)=(e\wedge h)\vee \lnot h=e\wedge h$, because $\lnot h\leq n\leq e\wedge
h$, and in the same way we have $f_{h}^{\lnot h}(e)=h\wedge e$. The fourth
line follows because, by definition, $\overline{v}(e|h)=\overline{v}%
(f_{h}^{\lnot h}(e))\cdot \frac{1}{v(h)_{0}}$.
\end{proof}

The domain of application of this result is narrowed by the hypothesis $e$, $%
h\in \nabla $: this means, in terms of partial events, that we should limit
ourselves to test hypothesis of the kind $(H,\emptyset )$ on the basis of
data like $(E,\emptyset )$. In both cases we have a particular kind of
partial event, an event that may occur only positively and cannot ever occur
negatively. In the following we shall prove a result in which $h$ and $e$
are free from any restriction.

Firstly we observe that, for any $e\in A$, there are three kinds of elements
naturally related to $e$: i) $e^{+}=e\vee n$, the positive part of $e$, ii) $%
e^{-}=\lnot e\vee n$, the negative part of $e$, iii) $\nabla e=e\vee \lnot e$%
, the join of the positive and the negative part. We note that $e^{+}$, $%
e^{-}$ and $\nabla e$ are all elements of $\nabla $. In the theorem we are
going to prove, three partial probabilities are related: i) $v(h|e)$, the
probability of the hypothesis $h$, given the occurrence of the positive
cases of the data, ii) $\overline{v}(h|e^{-})$, the probability of the
hypothesis $h$, given the occurrence of the negative cases of the data, iii) 
$\overline{v}(h|\nabla e)$, the probability of the hypothesis $h$, given the
occurrence of all possible cases of the data. The following lemma gives a
sufficient condition for the existence of such probabilities.

\begin{lemma}
\label{lemma1}If $\mathcal{A}$ is a DMF-algebra, $\overline{v}$ is a partial
valuation on $\mathcal{A}$ and $a\in A$, then

\begin{enumerate}
\item  $\overline{v}(a^{+})_{0}=\overline{v}(a)_{0}$, $\overline{v}%
(a^{-})_{0}=\overline{v}(a)_{1}$ and $\overline{v}(\nabla a)_{0}=\overline{v}%
(a)_{0}+\overline{v}(a)_{1},$

\item  $\overline{v}(a)_{0}$, $\overline{v}(a)_{1}\neq 0$ is a sufficient
condition for the existence of the conditional partial valuations $\overline{%
v}(-|a^{+})$, $\overline{v}(-|a^{-})$ and $\overline{v}(-|\nabla a)$.
\end{enumerate}
\end{lemma}

\begin{proof}
1. We have $\overline{v}(a^{+})_{0}=\overline{v}(a\vee n)_{0}=\overline{v}%
(a)_{0}$ and $\overline{v}(a^{-})_{0}=\overline{v}(\lnot a\vee n)_{0}=%
\overline{v}(a)_{1}$, by point 5 of theorem \ref{teopropr_v}. We have also $%
\overline{v}(\nabla a)_{0}=\overline{v}(a\vee \lnot a)_{0}=\overline{v}%
(a)_{0}+\overline{v}(\lnot a)_{0}+\overline{v}(a\wedge \lnot a)_{0}=%
\overline{v}(a)_{0}+\overline{v}(a)_{1}$, because $\overline{v}(a\wedge
\lnot a)=(0,y)$, for some $y\in \lbrack 0,1]$.

2. As we have seen at the end of par. \ref{parrpv}, three conditions are to
be satisfied for the existence of $\overline{v}(-|a^{+})$: i) $a^{+}\in
\nabla $, ii) $\overline{v}$ must be isotone, iii) $\overline{v}%
(a^{+})_{0}\neq 0$. The first follows by definition of $a^{+}$, the second \
follows by hypothesis, the third follows from $\overline{v}(a)_{0}\neq 0$
and point 1) above. The same kind of argument \ works for the existence of $%
\overline{v}(-|a^{-})$, we only remark that $a^{-}\in \nabla $ follows by
definition and $\overline{v}(a^{-})_{0}\neq 0$ follows from $\overline{v}%
(a)_{1}\neq 0$ and point 1) above. As for $\overline{v}(-|\nabla a)$, we
observe that $\nabla a\in \nabla $ holds by definition and $\overline{v}%
(\nabla a)_{0}\neq 0$ follows from $\overline{v}(a)_{0}$, $\overline{v}%
(a)_{1}\neq 0$ and point 1) above.
\end{proof}

The following lemma shows some properties of $f_{e^{+}}^{\lnot e^{+}}$, $%
f_{e^{-}}^{\lnot e^{-}}$ \ and $f_{\nabla e}^{\lnot \nabla e}$. (To save
notation, we write $\lnot e^{+}$ in place of $\lnot (e^{+})$ and $\lnot e^{-}$
in place of $\lnot (e^{-})$.)

\begin{lemma}
If $\mathcal{A}$ is a DMF-algebra and $\overline{v}$ is a partial valuation
on $\mathcal{A}$, then for every $e$, $a\in A$:

\begin{enumerate}
\item  $f_{e^{+}}^{\lnot e^{+}}(a)\vee f_{e^{-}}^{\lnot e^{-}}(a)=f_{\nabla
e}^{\lnot \nabla e}(a)\vee n$,

\item  $f_{e^{+}}^{\lnot e^{+}}(a)\wedge f_{e^{-}}^{\lnot
e^{-}}(a)=f_{\nabla e}^{\lnot \nabla e}(a)\wedge n$

\item  $\overline{v}(f_{\nabla e}^{\lnot \nabla e}(a))=\overline{v}%
(f_{e^{+}}^{\lnot e^{+}}(a))+\overline{v}(f_{e^{-}}^{\lnot e^{-}}(a))$.
\end{enumerate}
\end{lemma}

\begin{proof}
1. We remember that $e^{+}=e\vee n$, $e^{-}=\lnot e\vee n$, so $\lnot
e^{+}=\lnot e\wedge n$ and $\lnot e^{-}=e\wedge n$, thus 
\begin{eqnarray*}
f_{e^{+}}^{\lnot e^{+}}(a)\vee f_{e^{-}}^{\lnot e^{-}}(a)&=&
((a\wedge e^{+})\vee \lnot e^{+})\vee ((a\wedge e^{-})\vee \lnot e^{-}) \\
&=&((a\wedge (e\vee n))\vee (\lnot e\wedge n))\vee ((a\wedge (\lnot e\vee
n))\vee (e\wedge n))) \\
&=&(a\wedge e)\vee (\lnot e\wedge n)\vee (a\wedge \lnot e)\vee (e\wedge n)\vee
(a\wedge n) \\
&=&(a\wedge (e\vee \lnot e))\vee (n\wedge (e\vee \lnot e))\vee (a\wedge n) \\
&=&(a\wedge (e\vee \lnot e))\vee n.
\end{eqnarray*}
On the other side, 
\begin{eqnarray*}
f_{\nabla e}^{\lnot \nabla e}(a)\vee n&=&((a\wedge \nabla e)\vee \lnot \nabla
e)\vee n \\
&=&((a\wedge (e\vee \lnot e))\vee (e\wedge \lnot e))\vee n \\
&=&((a\wedge (e\vee \lnot e))\vee n.
\end{eqnarray*}

2. 
\begin{eqnarray*}
f_{e^{+}}^{\lnot e^{+}}(a)\wedge f_{e^{-}}^{\lnot e^{-}}(a)&=&\lnot (\lnot
f_{e^{+}}^{\lnot e^{+}}(a)\vee \lnot f_{e^{-}}^{\lnot e^{-}}(a)) \\
&=&\lnot (f_{e^{+}}^{\lnot e^{+}}(\lnot a)\vee f_{e^{-}}^{\lnot e^{-}}(\lnot
a)) \\
&=&\lnot (f_{\nabla e}^{\lnot \nabla e}(\lnot a)\vee n) \\
&=&f_{\nabla e}^{\lnot \nabla e}(a)\wedge n,
\end{eqnarray*}
where line two and four follow because $f_{e^{+}}^{\lnot e^{+}}$ and $%
f_{\nabla e}^{\lnot \nabla e}$ are morphisms of DMF-algebras, by theorem \ref
{teorelpar}.

3. By point 1) and 2) above, we have 
\begin{eqnarray*}
\overline{v}(f_{e^{+}}^{\lnot e^{+}}(a)\vee (f_{e^{-}}^{\lnot e^{-}}(a))&=&%
\overline{v}(f_{\nabla e}^{\lnot \nabla e}(a)\vee n) \\
&=&\overline{v}(f_{\nabla e}^{\lnot \nabla e}(a))-\overline{v}(f_{\nabla
e}^{\lnot \nabla e}(a)\wedge n) \\
&=&\overline{v}(f_{\nabla e}^{\lnot \nabla e}(a))-\overline{v}%
(f_{e^{+}}^{\lnot e^{+}}(a)\wedge f_{e^{-}}^{\lnot e^{-}}(a)),
\end{eqnarray*}
thus 
\begin{eqnarray*}
\overline{v}(f_{\nabla e}^{\lnot \nabla e}(a))&=&\overline{v}(f_{e^{+}}^{\lnot
e^{+}}(a)\vee (f_{e^{-}}^{\lnot e^{-}}(a))+\overline{v}(f_{e^{+}}^{\lnot
e^{+}}(a)\wedge (f_{e^{-}}^{\lnot e^{-}}(a)) \\
&=& \overline{v}(f_{e^{+}}^{\lnot e^{+}}(a))+\overline{v}(f_{e^{-}}^{\lnot
e^{-}}(a))-\overline{v}(f_{e^{+}}^{\lnot e^{+}}(a)\wedge (f_{e^{-}}^{\lnot
e^{-}}(a))+ \\
&+&\overline{v}(f_{e^{+}}^{\lnot e^{+}}(a)\wedge (f_{e^{-}}^{\lnot e^{-}}(a))
\\
&=&\overline{v}(f_{e^{+}}^{\lnot e^{+}}(a))+\overline{v}(f_{e^{-}}^{\lnot
e^{-}}(a)).
\end{eqnarray*}
\end{proof}

For any $a\in A$, we define the \textit{bias} of $a$ as the ratio $\theta
(a)=\overline{v}(a)_{1}/\overline{v}(a)_{0}$. The number $\theta (a)$
measures the inclination of the event $a$ toward coming into existence. When $\theta
(a)<1$, the event shows an inclination toward happening, when $\theta (a)>1$
the event shows an inclination toward not-happening and when $\theta (a)=1$
the event shows no propensity.

\begin{theorem}
If $\mathcal{A}$ is a DMF-algebra and $\overline{v}$ is an isotone partial
valuation on $\mathcal{A}$, then for every $e$, $h\in A$ such that $%
\overline{v}(e)_{0}$, $\overline{v}(e)_{1}\neq 0$, 
\begin{equation*}
\overline{v}(h|e^{+})=\overline{v}(h|\nabla e)\cdot (1+\theta (e))-\overline{%
v}(h|e^{-})\cdot \theta (e).
\end{equation*}
\end{theorem}

\begin{proof}
From our hypothesis we see that the existence of the conditional partial
valuations $\overline{v}(-|e^{+})$, $\overline{v}(-|e^{-})$ and $\overline{v}%
(-|\nabla e)$ is guaranteed by point 2) of lemma \ref{lemma1}. So we have 
\begin{eqnarray*}
\overline{v}(h|e^{+}) &=&\overline{v}(f_{e^{+}}^{\lnot e^{+}}(h))\cdot \frac{%
1}{\overline{v}(e^{+})_{0}} \\
&=&(\overline{v}(f_{\nabla e}^{\lnot \nabla e}(h))-\overline{v}%
(f_{e^{-}}^{\lnot e^{-}}(h)))\cdot \frac{1}{\overline{v}(e^{+})_{0}} \\
&=&(\overline{v}(h|\nabla e)\cdot \overline{v}(\nabla e)_{0}-\overline{v}%
(h|e^{-})\cdot \overline{v}(e^{-})_{0})\cdot \frac{1}{\overline{v}(e^{+})_{0}%
},
\end{eqnarray*}
where the first line follows by definition of $\overline{v}(-|e^{+})$, the
second line by point 3) of the preceding lemma and the third line because 
\begin{equation*}
\overline{v}(h|e^{-})=\overline{v}(f_{e^{-}}^{\lnot e^{-}}(h))\cdot \frac{1}{%
\overline{v}(e^{-})_{0}}\text{ and }\overline{v}(h|e^{-})=\overline{v}%
(f_{\nabla e}^{\lnot \nabla e}(a)(h))\cdot \frac{1}{\overline{v}(\nabla
e)_{0}}
\end{equation*}
by definition of $\overline{v}(-|e^{-})$ and of $\overline{v}(-|\nabla e)$.
Thus, by point 1) of lemma \ref{lemma1} 
\begin{eqnarray*}
\overline{v}(h|e^{+}) &=&(\overline{v}(h|\nabla e)\cdot (\overline{v}(e)_{0}+%
\overline{v}(e)_{1})-\overline{v}(h|e^{-})\cdot \overline{v}(e)_{1})\cdot 
\frac{1}{\overline{v}(e)_{0}} \\
&=&(\overline{v}(h|\nabla e)\cdot \frac{\overline{v}(e)_{0}+\overline{v}%
(e)_{1}}{\overline{v}(e)_{0}}-\overline{v}(h|e^{-})\cdot \frac{\overline{v}%
(e)_{1}}{\overline{v}(e)_{0}} \\
&=&\overline{v}(h|\nabla e)\cdot (1+\theta (e))-\overline{v}(h|e^{-})\cdot
\theta (e).
\end{eqnarray*}
\end{proof}

\section{Partial probability of sentences\label{parprobparsent}}

As we have seen in par. \ref{parprobsent}, we can understand classic
probability as a degree of belief in a sentence. We can do the same thing
with partial probability: we have only to shift from bivalent logic to
Kleene's logic and from probability values in $[0,1]$ to probability values
in $T$. This should justify a brief digression in the semantics of Kleene's
logic.

The language of Kleene's \ $n$-ary logic is $L_{n}^{\ast }=$ $L_{n}\cup
\{n\} $, where $L_{n}$ is the $n$-ary language of classical logic introduced
in par. \ref{parprobsent}. We denote with $F_{n}^{\ast }$ the set of
formulas of $L$ and with $\mathcal{F}_{n}^{\ast }$ the algebra of $n$-ary
formula. We write simply $L^{\ast }$, $%
F^{\ast }$ and $\mathcal{F}^{\ast }$ when no confusion is possible. We
denote with $K$ the set $\{0,n,1\}$ of the truth-values of Kleene's logic,
leaving to the reader the task of distinguishing $n$ as a symbol of the
formal language from $n$ as a truth-value. The meanings of $n$-ary formulas
are to be found as elements of the field of partial set

\begin{equation*}
\mathcal{D}(K^{n})=(D(K^{n}),\sqcap ,\sqcup ,-,(\emptyset
,K^{n}),(K^{n},\emptyset ),(\emptyset ,\emptyset )).
\end{equation*}
Every $s\in K^{n}$ cas be seen as an instantaneous description of the world,
at the level of the atomic facts represented by sentential variables. When $%
s_{i}=n$, the atomic fact represented by $p_{i}$ is neither happened nor
not-happened. This uncertainty may be of an epistemic kind, related to a
lack of knowledge, or may be deeply rooted in the reality.

We associate to every $n$-ary formula $\alpha $ a \textit{meaning} $M(\alpha
)$ as an element of $\mathcal{D}(K^{n})$ as follows. Firstly, we define a
function $g:\{p_{i}:i<n\}\rightarrow \mathcal{D}(K^{n})$ setting

\begin{equation*}
g(p_{i})=(\{s\in K^{n}:s_{i}=1\},\{s\in K^{n}:s_{i}=0\}).
\end{equation*}
Then a function $M:\mathcal{F}\rightarrow \mathcal{D}(K^{n})$ can be
defined, as the only morphism induced by $g$, as follows:

\begin{eqnarray*}
M(\alpha \wedge \beta ) &=&M(\alpha )\sqcap M(\beta ), \\
M(\alpha \vee \beta ) &=&M(\alpha )\sqcup M(\beta ), \\
M(\lnot \alpha ) &=&-M(\alpha ), \\
M(0) &=&(\emptyset ,K^{n}), \\
M(1) &=&(K^{n},\emptyset ), \\
M(n) &=&(\emptyset ,\emptyset ).
\end{eqnarray*}
Thus the meaning of $\alpha $ is a partial set $M(\alpha )$, where $M(\alpha
)_{0}$ and $M(\alpha )_{1}$ are respectively the positive and the negative
models of $\alpha $.

The semantics can also be given through a function $V_{s}$ that assigns to
every formula $\alpha $ a truth value $V_{s}(\alpha )\in K$ with respect to
a possible world $s\in K^{n}$. Firstly we denote with $\mathcal{K}$ the
algebra $(K,\wedge ,\vee ,\lnot ,0,1,n)$, where the operations are defined
as follows: $x\wedge y=\min (x,y)$ and $x\vee y=\max (x,y)$, supposing $K$
linearly ordered by $0\leq n\leq 1$. As for negation, we set $\lnot (n)=n$, $%
\lnot (0)=1$, $\lnot (1)=0$. Obviously, $\mathcal{K}$ is a DMF-algebra. Then
we define the assignment of truth
values $h_{s}:\{p_{i}:i<n\}\rightarrow K$ by $h_{s}(p_{i})=s_{i}$.
Finally, $V_{s}$ is the only morphism from $\mathcal{F}$ to $\mathcal{K}$
induced by $h_{s}$. We say that $\alpha $ is \textit{true in} $s$ iff $%
V_{s}(\alpha )=1$, \textit{false} if $V_{s}(\alpha )=0$ and \textit{undefined%
} if $V_{s}(\alpha )=n$.

These two ways of giving a semantics are equivalent: if we take the notion
of meaning given by $M$ as primitive, then we can define 
\begin{equation*}
V_{s}(\alpha )=\left\{ 
\begin{array}{ccc}
0 & \mathrm{if} & s\in M(\alpha )_{1}, \\ 
1 & \mathrm{if} & s\in M(\alpha )_{0}, \\ 
n & \text{\textrm{if}} & s\in K^{n}-(M(\alpha )_{0}\cup M(\alpha )_{1});
\end{array}
\right.
\end{equation*}
if we take the notion of truth in the possible world $s$ given by $V_{s}$ as
primitive, then 
\begin{equation*}
M(\alpha )=(\{s\in K^{n}:V_{s}(\alpha )=1\},\{s\in K^{n}:V_{s}(\alpha )=0\}).
\end{equation*}

We define the notion of \textit{logical consequence} as follows: $\alpha
\models \beta $ iff $M(\alpha )\sqsubseteq M(\beta )$ iff $M(\alpha
)_{0}\subseteq M(\beta )_{0}$ and $M(\beta )_{1}\subseteq M(\alpha )_{1}$
iff every postive model of $\alpha $ is a positive model of $\beta $ and
every negative model of $\beta $ is a negative model of $\alpha $. The
notion of logical consequence can be generalized to 
\begin{equation*}
\Gamma \models \alpha \text{ iff }{\sqcap }\{M(\gamma ):\gamma \in \Gamma
\}\sqsubseteq M(\alpha ).
\end{equation*}
In terms of $V_{s}$, the definition runs as follows: 
\begin{equation*}
\Gamma \models \alpha \text{ iff, for all }s\in K^{n}\text{, }\bigwedge
\{V_{s}(\gamma ):\gamma \in \Gamma \}\leq V_{s}(\alpha ).
\end{equation*}

Now we can introduce probability as a degree of belief in a sentence as
follows. We say that $\pi :F_{n}^{\ast }\rightarrow T$ is a \textit{partial
probability function on} $L_{n}^{\ast }$ if the following axioms are
satisfied, where $\models $ denotes logical consequence in Kleene's logic:

\begin{enumerate}
\item  $1\models \alpha $ implies $\pi (\alpha )=(1,0)$,

\item  $\pi (\alpha \vee \beta )=\pi (\alpha )+\pi (\beta )-\pi (\alpha
\wedge \beta )$,

\item  $\pi (\lnot \alpha )=\sigma (\pi (\alpha ))$,

\item  $n\models \alpha $ implies $(0,0)\preccurlyeq \pi (\alpha ).$
\end{enumerate}

The following theorem shows some fundamental properties of $\pi $.

\begin{theorem}
\label{teoprobfun}If $\pi $ is a partial probability function on $%
L_{n}^{\ast }$, then

\begin{enumerate}
\item  $\pi (n)=(0,0).$

\item  $\alpha \models 0$ implies $\pi (\alpha )=(0,1)$,

\item  $\alpha \models n$ implies $\pi (\alpha )\preccurlyeq (0,0)$,

\item  $\pi (\alpha )=\pi (\alpha \vee n)+\pi (\alpha \wedge n)$.
\end{enumerate}
\end{theorem}

\begin{proof}
1. From $n\models n$ we have $(0,0)\preccurlyeq \pi (n)$, by axiom 4. In
Kleene's logic we have $n\models \lnot n$ and so $(0,0)\preccurlyeq \pi
(\lnot n)$. In general we have $(x,y)\preccurlyeq (x^{\prime },y^{\prime })$
iff $\sigma (x^{\prime },y^{\prime })\preccurlyeq \sigma (x,y)$, thus $%
\sigma (\pi (\lnot (n)))\preccurlyeq (0,0)$. By axiom 3, $\sigma (\pi (\lnot
(n)))=\sigma (\sigma (\pi (n)))=\pi (n)$ holds. Thus $\pi (n)\preccurlyeq
(0,0)$.

2. In Kleene's logic $\alpha \models 0$ implies $1\models \lnot \alpha $ and
so, by axiom 1, $\pi (\lnot \alpha )=(1,0)$. Thus, by axiom 3, $\sigma (\pi
(\alpha ))=(1,0)$ e $\pi (\alpha )=(0,1)$.

3. In Kleene's logic $\alpha \models n$ implies $n\models \lnot \alpha $, so 
$(0,0)\preccurlyeq \pi (\lnot \alpha )=\sigma (\pi (\alpha ))$ and $\pi
(\alpha )\preccurlyeq (0,0)$.

4. $\pi (\alpha \vee n)=\pi (\alpha )+\pi (n)-\pi (\alpha \wedge n)=\pi
(\alpha )-\pi (\alpha \wedge n)$ by axiom 2 and point 1).
\end{proof}

Now we can prove two translatability results, as we did in par. \ref{partrad}
for the Boolean framework: partial probability as a measure of a partial set
and partial probability as a degree of belief in a sentence can be
translated one into the other.

\section{From sentences to partial sets\label{partradpar1}}

The main tool in this translation is the notion of Lindenbaum algebra for
Kleene's logic. As in Boolean logic, the Lindenbaum algebra arises from the
identification of logically equivalent formulas. We define a $2$-ary
relation $\sim $ on $F_{n}^{\ast }$ setting $\alpha \sim \beta $ iff $%
M(\alpha )=M(\beta )$, where $M:\mathcal{F}_{n}\rightarrow \mathcal{D}(K^{n})$
is the morphism defined in the preceding paragraph. The quotient $\mathcal{F}%
_{n}^{\ast }/\sim $ is the \textit{Lindenbaum algebra of Kleene's} $n$-%
\textit{ary logic}. If we denote with $M[\mathcal{F}_{n}^{\ast }]$ the image
of $\mathcal{F}_{n}^{\ast }$\ in $\mathcal{D}(K^{n})$, we have that $%
\mathcal{F}_{n}^{\ast }/\sim $ is isomorphic to $M[\mathcal{F}_{n}^{\ast }]$%
. As a subalgebra of $\mathcal{D}(K^{n})$, $M[\mathcal{F}_{n}^{\ast }]$ is a
DMF-algebra and so is $\mathcal{F}_{n}^{\ast }/\sim $.

Now we suppose that partial probability be given as a partial probability
function $\pi $ on $L_{n}^{\ast }$. We aim to define a partial probability
space $(A,\mathcal{G}_{A},\mu )$ and a function $\varphi :F_{n}^{\ast
}\rightarrow \mathcal{G}_{A}$ such that $\pi (\alpha )=\mu (\varphi (\alpha
))$ holds for all $\alpha \in F_{n}^{\ast }$. The translation from partial
probability on sentences to partial probability on partial sets suffers from
a drawback: it works only for isotone functions of partial probability. We
say that $\pi $ is \textit{isotone} if $\alpha \models \beta $
implies $\pi (\alpha )\preceq \pi (\beta )$. Every isotone function $\pi $
is obviously compatible with $\sim $: if $\alpha \sim \beta $ then $\alpha
\models \beta $ and $\beta \models \alpha $ so $\pi (\alpha )=\pi (\beta )$
by isotonicity of $\pi $ and antisimmetry of $\preceq $.

\begin{lemma}
If $\pi $ is an isotone partial probability function on $L_{n}^{\ast }$,
then the function $\pi ^{\ast }$ from $\mathcal{F}_{n}/\sim $ to $T$ defined
setting $\pi ^{\ast }(|\alpha |)=\pi (\alpha )$ is a partial valuation on
the DMF-algebra $\mathcal{F}_{n}/\sim $.
\end{lemma}

\begin{proof}
The function $\pi ^{\ast }$ is well-defined because $\pi $ is compatible
with $\sim $. We show that $\pi ^{\ast }$ satisfies the axioms of partial
valuation.

1. $\pi ^{\ast }(|0|)=\pi (0)=(0,1)$, by point 2 of theorem \ref{teoprobfun}.

2. $\pi ^{\ast }(|\alpha |\vee |\beta |)=\pi ^{\ast }(|\alpha \vee \beta
|)=\pi (\alpha \vee \beta )=\pi (\alpha )+\pi (\beta )-\pi (\alpha \wedge
\beta )=$ $\pi ^{\ast }(|\alpha |)+\pi ^{\ast }(|\beta |)-\pi ^{\ast
}(|\alpha |\wedge |\beta |)$, by axiom 2.

3. $\pi ^{\ast }(\lnot |\alpha |)=\pi ^{\ast }(|\lnot \alpha |)=\pi (\lnot
\alpha )=\sigma (\pi (\alpha ))=\sigma (\pi ^{\ast }(|\alpha |))$, by axiom
3.

4. If $|n|\leq |\alpha |$ in the Lindenbaum algebra, then $n\models \alpha $%
, by definition of $\leq $ in the Lindenbaum algebra, thus $%
(0,0)\preccurlyeq \pi (\alpha )=\pi ^{\ast }(|\alpha |)$.
\end{proof}

\begin{theorem}
If $\pi $ is an isotone partial probability function on $L_{n}$, then there
is a partial probability space $(K^{n},M[\mathcal{F}_{n}],\mu )$ and a
morphism $\varphi $ from $\mathcal{F}_{n}$ to $M[\mathcal{F}_{n}]$ such
that, for all $\alpha \in L_{n}$, $\pi (\alpha )=\mu (\varphi (\alpha ))$.
\end{theorem}

\begin{proof}
We know that an isomorphism $\psi :\mathcal{F}_{n}/\sim \rightarrow M[%
\mathcal{F}_{n}]$ is naturally associated to the morphism $M:\mathcal{F}%
_{n}\rightarrow \mathcal{D}(K^{n})$. Thus we define a function $\mu $ from $%
M[\mathcal{F}_{n}]$ to $T$ setting, for all partial set $(X,Y)$ in $M[%
\mathcal{F}_{n}]$, $\mu (X,Y)=\pi ^{\ast }(\psi ^{-1}(X,Y))$, where $\pi
^{\ast }$ is defined as in the above lemma. By theorem \ref{teo333}, $\mu $
is a partial valuation on $M[\mathcal{F}_{n}]$, as it comes from the
composition of the morphism $\psi ^{-1}$ with the valuation $\pi ^{\ast }$.
So we can define $\varphi $ from $\mathcal{F}_{n}$ to $M[\mathcal{F}_{n}]$,
setting $\varphi (\alpha )=\psi (|\alpha |)$: it can be easily shown that $%
\varphi $ is a morphism such that 
\begin{equation*}
\mu (\varphi (\alpha ))=\mu (\psi (|\alpha |)) 
=\pi ^{\ast }(\psi ^{-1}(\psi (|\alpha |))) 
=\pi ^{\ast }(|\alpha |) 
=\pi (\alpha ),
\end{equation*}
as required by the theorem.
\end{proof}

\section{From partial sets to sentences\label{partradpar2}}

Finally, we face the problem of translating probability conceived as a
valuation on a field of partial sets, into probability as a partial
probability function on the formulas of a formal language. The main tool in
this translation is the freeness of Lindenbaum algebras on the class of
DMF-algebras. The proof of this result requires a short digression on the
properties of ideals and filters in DMF-algebras.

Firstly, we recall the prime ideal theorem (for a proof, see for instance 
\cite{davey1990}, theorem 9.13).

\begin{theorem}
If $\mathcal{A}$ is a distributive lattice and $I$ and $F$ are respectively
an ideal and a filter such that $I\cap F=\emptyset $, then

\begin{enumerate}
\item  there is a prime ideal $J$ such that $I\subseteq J$ and $J\cap
F=\emptyset ,$

\item  there is a prime filter $G$ such that $F\subseteq G$ and $G\cap
I=\emptyset $.
\end{enumerate}
\end{theorem}

The following corollary about separating points in distributive lattices
will be useful.

\begin{corollary}
\label{coridpr}If $x$ and $y$ belong to a distributive lattice $\mathcal{A}$
and $x\nleqslant y$, then there are a prime ideal $I$ and a prime filter $F$
in $\mathcal{A}$ such that: i) $I\cap F=\emptyset $; ii) $y\in I$ and $%
x\notin I$; iii) $x\in F$ and $y\notin F$.
\end{corollary}

\begin{proof}
Let $\downarrow y$ be the ideal $\{a\in A:a\leq y\}$ and $\uparrow x$ be the
filter $\{a\in A:x\leq a\}$. As $\downarrow y$ and $\uparrow x$ are disjoint
by the hypothesis $x\nleqslant y$, by the above theorem there is a prime
ideal $I$ such that $\downarrow y\subseteq I$ and $I\cap \uparrow
x=\emptyset $, so $y\in I$ and $x\notin I$. As $I$ and $\uparrow x$ are
disjoint, there is a prime filter $F$ such that $\uparrow x\subseteq F$ and $%
F\cap I=\emptyset $: so $x\in F$ and $y\notin F$.
\end{proof}

Now we show a similar theorem about separating points in DMF-algebras.
Whereas in distributive lattices a single ideal (filter) separates $x$ from $%
y$, in DMF-algebras points are set apart by a pair (ideal, filter). In the
following, we denote with $\lnot X$ the set $\{\lnot x:x\in X\}$. The
following lemma is left to the reader.

\begin{lemma}
\label{lemma38}
If $\mathcal{A}$ is a DM-algebra, then:

\begin{enumerate}
\item  $X$ is an ideal iff $\lnot X$ is a filter,

\item  $X$ is a filter iff $\lnot X$ is an ideal,

\item  $X$ is a prime ideal filter iff $\lnot X$ is a prime filter (ideal).
\end{enumerate}
\end{lemma}

\begin{lemma}
\label{lemma39}
In every DMF-algebra, for all ideal $I$, $n\notin I$ iff $I\cap \lnot
I=\emptyset $.
\end{lemma}

\begin{proof}
In one direction, $n\in I$ implies $\lnot n\in \lnot I$, but $n=\lnot n$, so 
$I\cap \lnot I\neq \emptyset $. In the other direction, we suppose $x\in
I\cap \lnot I$. Then $x\in \lnot I$ implies $x=\lnot i$, for some $i\in I$,
and from $x\in I$ we get $\lnot i\in I$ and then $i\vee \lnot i\in I$. As in
DMF-algebras $n\leq i\vee \lnot i$, we have $n\in I$.
\end{proof}

\begin{theorem}
\label{tifdmf}If $\mathcal{A}$ is a DMF-algebra and $a\nleqslant b$, then
there is pair $(G,H)$ such that:

\begin{enumerate}
\item  $G$ is a prime ideal in $\mathcal{A}$ and $H$ is a prime filter in $%
\mathcal{A}$, with $G\cap H=\emptyset $,

\item  $H=\lnot G$,

\item  $a\notin G$ and $b\in G$ or $a\in H$ and $b\notin H$.
\end{enumerate}
\end{theorem}

\begin{proof}
By corollary \ref{coridpr}, there are a prime ideal $I$ and a prime filter $F
$ such that: i) $I\cap F=\emptyset $; ii) $b\in I$ and $a\notin I$; iii) $%
a\in F$ and $b\notin F$. If $n\notin I$ then $I\cap \lnot I=\emptyset $, by
lemma \ref{lemma39}. By lemma \ref{lemma38}, $\lnot I$ is a prime filter, so
$(I,\lnot I)$ is the pair $(G,H)$ we are looking for. If $n\in I$ then 
$n\notin F$, because $I\cap F=\emptyset $. As $\lnot F$ is a prime ideal, by
lemma \ref{lemma38}, $(\lnot F,F)$ is the pair $(G,H)$ we are looking for.
\end{proof}

We need a last result concerning the relations between prime filters and
epimorphisms on $\mathcal{K}$ in DMF-algebras. We know that there is a tight
connection between prime filters in bounded lattices and epimorphisms $%
\varphi :\mathcal{A}\rightarrow \mathbf{2}$, where $\mathbf{2}$ \ is the
two-element lattice. On one side, if $\varphi $ is such an epimorphism, then 
$I_{\varphi }=\varphi ^{-1}\{0\}$ is a prime ideal of $\mathcal{A}$. On the
other side, for every prime ideal $I$ of $\mathcal{A}$, the function $%
\varphi _{I}:$ $\mathcal{A}\rightarrow 2$ defined by

\begin{center}
$\varphi _{I}(a)=\left\{ 
\begin{array}{ccc}
0 & \mathrm{if} & a\in I \\ 
1 & \mathrm{if} & a\notin I
\end{array}
\right. $
\end{center}

\noindent is an epimorphism. (See \cite[ex. 9.2]{davey1990}.) If $%
\mathcal{A}$ is a Boolean algebra then $\varphi _{I}$ is a Boolean morphism.
A dual theorem holds for filters. Now we suppose that $\mathcal{A}$ be a
DMF-algebra and $I$ a prime ideal of $\mathcal{A}$ such that $n\notin I$.
By theorem \ref{lemma39}, $I$ and $\lnot I$ are disjoint, so we can define a function
$\varphi _{I}:A\rightarrow K$ setting

\begin{equation*}
\varphi _{I}(a)=\left\{ 
\begin{array}{ccc}
0 & \mathrm{if} & a\in I, \\ 
1 & \mathrm{if} & a\in \lnot I, \\ 
n & \mathrm{if} & a\in A-(I\cup \lnot I).
\end{array}
\right.
\end{equation*}

\begin{theorem}
\label{tepi3} If $\mathcal{A}$ is a DMF-algebra and $I$ a prime ideal of $%
\mathcal{A}$ such that $n\notin I$, then $\varphi _{I}:A\rightarrow K$ is an
epimorphism of DMF-algebras.
\end{theorem}

\begin{proof}
Obviously $\varphi _{I}(0)=0$, $\varphi _{I}(1)=1$ and $\varphi _{I}(n)=n$,
because $n\notin I$ and $n\notin \lnot I$.

We show that $\varphi _{I}$ preserves $\wedge $. If $\varphi _{I}(x\wedge
y)=0$ then $x\wedge y\in I$ and $x\in I$ or $y\in I$, because $I$ is prime.
Thus $\varphi _{I}(x)=0$ or $\varphi _{I}(y)=0$ and so $\varphi
_{I}(x)\wedge \varphi _{I}(y)=0$. If $\varphi _{I}(x\wedge y)=1$ then $%
x\wedge y\in \lnot I$ and both $x$ and $y$ belong to $\lnot I$, because $%
\lnot I$ is a filter. Then $\varphi _{I}(x)=\varphi _{I}(y)=1$ and $\varphi
_{I}(x)\wedge \varphi _{I}(y)=1$. Finally we suppose $\varphi _{I}(x\wedge
y)=n$, then $x\wedge y\notin I$ and $x\wedge y\notin \lnot I$. As $I$ is an
ideal, we have $x\notin I$ and $y\notin I$, otherwise we could derive $%
x\wedge y\in I$. As $\lnot I$ is a filter, we have $x\notin \lnot I$ or $%
y\notin \lnot I$, otherwise we could derive $x\wedge y\in \lnot I$. We can
distinguish the following three cases. Case 1, both $x$ and $y$ are in $%
A-(I\cup \lnot I)$. Then $\varphi _{I}(x)=\varphi _{I}(y)=n$ and so $\varphi
_{I}(x)\wedge \varphi _{I}(y)=n$. Case 2, $x$ is in $A-(I\cup \lnot I)$ and $%
y$ in $\lnot I$. Then $\varphi _{I}(x)=n$, $\varphi _{I}(y)=1$ and so $%
\varphi _{I}(x)\wedge \varphi _{I}(y)=n$. Case 3, $y$ is in $A-(I\cup \lnot
I)$ and $x$ in $\lnot I$. Then $\varphi _{I}(x)=1$, $\varphi _{I}(y)=n$ and
so $\varphi _{I}(x)\wedge \varphi _{I}(y)=n$.

An analogous proof shows that $\vee $ is preserved.

Finally we show that $\varphi _{I}$ preserves $\lnot $. If $\varphi
_{I}(\lnot x)=0$ then $\lnot x\in I$ and so $x\in \lnot I$ and $\varphi
_{I}(x)=1$, i.e. $\varphi _{I}(\lnot x)=\lnot \varphi _{I}(x)$. If $\varphi
_{I}(\lnot x)=1$ then $\lnot x\in -I$ and so $x\in I$ e $\varphi _{I}(x)=0$,
i.e. $\varphi _{I}(\lnot x)=\lnot \varphi _{I}(x)$. If $\varphi _{I}(\lnot
x)=n$ then $\lnot x\notin I$ and $\lnot x\notin \lnot I$, thus $x\notin I$
and $x\notin \lnot I$. Then $\varphi _{I}(x)=n$. As $n=\lnot n$, we have $%
\varphi _{I}(\lnot x)=\lnot \varphi _{I}(x)$.
\end{proof}

\begin{theorem}
\label{teolindfree}The Lindenbaum algebra $\mathcal{F}_{n}^{\ast }/\sim $ is
free in the class of DMF-algebras, with $G=\{|p_{i}|:i\in n\}$ as a set of
free generators.
\end{theorem}

\begin{proof}
As the algebra of formulas $\mathcal{F}_{n}^{\ast }$ is generated by $%
P=\{p_{i}:i\in n\}$, $\mathcal{F}_{n}^{\ast }/\sim $ is generated by \ $%
G=\{|p_{i}|:i\in n\}$. We must show that $G$ is a set of free generators,
i.e. every function $g:G\rightarrow \mathcal{A}$, where $\mathcal{A}\in DMF$%
, can be extended to a unique morphism $\overline{g}:\mathcal{F}_{n}^{\ast
}/\sim \rightarrow \mathcal{A}$. Firstly, we define a function $%
f:P\rightarrow A$ setting $f(p_{i})=g(|p_{i}|)$. Then $f$ can be extended to
a unique morphism $\overline{f}:\mathcal{F}\rightarrow \mathcal{A}$ because $%
\mathcal{F}_{n}^{\ast }$ is the absolutely free algebra. Now we define $%
\overline{g}$ setting $\overline{g}(|\alpha |)=\overline{f}(\alpha )$. We
must show that the value of $\overline{g}$ is independent from the
representative of the equivalence class, i.e. $\alpha \sim \beta $ implies $%
\overline{f}(\alpha )=\overline{f}(\beta )$. This follows immediately, if
we can prove that
$\alpha \models \beta $ $\rightarrow \overline{f}(\alpha )\leq \overline{f}%
(\beta )$. In fact, if $\alpha \sim \beta $ then $\alpha \models \beta $ and 
$\beta \models \alpha $, so $\overline{f}(\alpha )=\overline{f}(\beta )$.

To prove that $\alpha \models \beta $ \ implies $\overline{f}(\alpha )\leq 
\overline{f}(\beta )$, we assume $\alpha \models \beta $ and suppose toward a
contradiction that $\overline{f}(\alpha )\nleq \overline{f}(\beta )$. Then,
by theorem \ref{tifdmf}, there is a pair $(I,F)$ in $\mathcal{A}$ such that: 
$I$ is a prime ideal, $F$ is a prime filter, $I\cap F=\emptyset $, $F=\lnot I
$ and finally $(\overline{f}(\alpha )\notin I$ and $\overline{f}(\beta )\in
I)$ or $(\overline{f}(\alpha )\in F$ and $\overline{f}(\beta )\notin F)$. We
observe that we must have $n\notin I$, otherwise, from $n\in I$ we could get 
$\lnot n\in \lnot I=F$ and $n\in I\cap F$, because $n=\lnot n$,
contradicting $I\cap F=\emptyset $. By theorem \ref{tepi3}, there is
a morphism $\varphi _{I}$:$\mathcal{A}\rightarrow K$ such that 
\begin{equation*}
\varphi _{I}(a)=\left\{ 
\begin{array}{ccc}
0 & \mathrm{if} & a\in I, \\ 
1 & \mathrm{if} & a\in \lnot I, \\ 
n & \mathrm{if} & a\in A-(I\cup \lnot I).
\end{array}
\right. 
\end{equation*}
Then $V=\varphi _{I}\circ \overline{f}$ is a morphism from $\mathcal{F}$ to $%
K$ such that: $(V(\beta )=0$ and $V(\alpha )\in \{0,1\})$ or $(V(\beta )\in
\{0,n\}$ and $V(\alpha )=1)$: in both cases we have $\alpha \nvDash \beta $,
what is absurd.

We must verify that $\overline{g}$ is a morphism extending $g$. We content
ourselves to verify that $\overline{g}$ preserves $\wedge $, the cases of
the other operations being analogous: 
\begin{equation*}
\overline{g}(|\alpha |\wedge |\beta |)=\overline{g}(|\alpha \wedge \beta |)
\\
=\overline{f}(\alpha \wedge \beta ) 
=\overline{f}(\alpha )\wedge \overline{f}(\beta ) 
=\overline{g}(\alpha )\wedge \overline{g}(\beta ).
\end{equation*}
Finally, $\overline{g}$ is an extension of $g$ because $\overline{g}%
(|p_{i}|)=\overline{f}(p_{i})=f(p_{i})=g(|p_{i}|)$.

The uniqueness of $\overline{g}$ follows as usual from the following fact of
general character: if $\mathcal{A}$ is generated by $G$ and both $f$ and $%
f^{\prime }$ are morphism from $\mathcal{A}$ to $\mathcal{B}$ coinciding on $%
G$, then $f=f^{\prime }$.
\end{proof}

Now we can give the translation from probability as a valuation on a field
of partial sets into probability as a partial probability function on
formulas. We suppose that a partial probability space $(A,\mathcal{G}%
_{A},\mu )$ be given and denote with $G_{A}$ the domain of $\mathcal{G}_{A}$%
. As the relation events/formulas is one/many, what we want is a
partial probability function $\pi $ and a function $\bar{\tau}$ from $%
\mathcal{G}_{A}$ to $P(F_{n}^{\ast })$ such that, for all $\alpha \in \bar{%
\varphi}(X,Y)$, $\mu (X,Y)=\pi (\alpha )$. The first step toward the
translation is the following theorem about valuations on Lindenbaum algebras.

\begin{theorem}
If $\overline{v}$ is a partial valuation on $\mathcal{F}_{n}^{\ast }/\sim $ then the
function $\pi :F_{n}^{\ast }\rightarrow T$ defined by $\pi (\alpha )=\overline{v}
(|\alpha |)$ is a partial probability function on $L_{n}$.
\end{theorem}

\begin{proof}
We show that the four axioms of partial probability functions are satisfied
by $\pi $.

1. If $1\models \alpha $ then $|\alpha |=1$ in $\mathcal{F}_{n}^{\ast }/\sim 
$ and then $\overline{v} (|\alpha |)=(1,0)$, by axiom 1 in the definition of partial
valuation of par. \ref{parvalpar}.

2. 
\begin{eqnarray*}
\pi (\alpha \vee \beta )&=&\overline{v} (|\alpha \vee \beta |) 
=\overline{v} (|\alpha |\vee |\beta |) 
=\overline{v} (|\alpha |)+\overline{v} (|\beta |)-\overline{v} (|\alpha |\wedge |\beta |) \\
&=&\overline{v} (|\alpha |)+\overline{v} (|\beta |)-\overline{v} (|\alpha \wedge \beta |)
=\pi (\alpha )+\pi (\beta )-\pi (\alpha \wedge \beta ).
\end{eqnarray*}

3. $\pi (\lnot \alpha )=\overline{v} (|\lnot \alpha |)=\overline{v} (\lnot |\alpha |)=\sigma
(\overline{v} (\left| \alpha \right| ))=\sigma (\pi (\alpha ))$, by axiom 3 in the
definition of partial valuation.

4. If $n\models \alpha $ then $|n|\leq |\alpha |$ in $\mathcal{F}$ $%
_{n}/\sim $ $\ $and by axiom 4 in the definition of partial valuation, $%
(0,0)\preccurlyeq \overline{v} (|\alpha |)$ so $(0,0)\preccurlyeq \pi (\alpha )$
\end{proof}

\begin{theorem}
Let $(A,\mathcal{G}_{A},\mu )$ be a partial probability space, with $%
A=\{a_{1},...,a_{n}\}$. Then there are a function $\bar{\tau}%
:G_{A}\rightarrow P(F)$ and a partial probability function $\pi $ on $L_{j}$
such that, for all partial event $(X,Y)$ in $G_{A}$ and all formula $\alpha $
in $\bar{\tau}(X,Y)$, $\mu (X,Y)=\pi (\alpha )$.
\end{theorem}

\begin{proof}
Let $j=\min \{|H|:H\subseteq G_{A}$, $H$ generates $\mathcal{G}_{A}\}$. We
denote with $H_{j}$ a set of generators of $\mathcal{G}_{A}$ of cardinal $j$%
. As $\mathcal{G}_{A}\subseteq \mathcal{D}(A)$ and $|\mathcal{D}(A)|=K^{n}$,
we have $j\leq K^{n}$. By theorem \ref{teolindfree}, $\mathcal{F}$ $%
_{j}/\sim $ is free on the class of DMF-algebras with $\{|p_{i}|:i\in j\}$
as a set of free generators, so every function $f:\{|p_{i}|:i<j\}\rightarrow 
\mathcal{G}_{A}$ can be extended to a morphism $\eta :\mathcal{F}$ $%
_{j}/\sim \rightarrow \mathcal{G}_{A}$. By theorem \ref{teo333} we obtain a
partial valuation $\overline{v} $ on $\mathcal{F}$ $_{j}/\sim $, setting 
$\overline{v} =\mu
\circ \eta $. By the preceding theorem, we can define a partial probability
function $\pi $ on $L_{j}$ setting $\pi (\alpha )=\overline{v} (|\alpha |)$.

As the cardinalities of $\{|p_{i}|:i<j\}$ and $H_{j}$ are identical, we can
choose a function $f$ that is a bijection between the generators of $%
\mathcal{F}$ $_{j}/\sim $ and the generators of $\mathcal{G}_{A}$. In this
way, the induced morphism $\eta $ will be an epimorphism. This enables the
following definition of $\bar{\tau}:G_{A}\rightarrow P(F)$: 
\begin{equation*}
\bar{\tau}(X,Y)=\bigcup \{|\xi |\in F_{j}/\sim :\eta (|\xi |)=(X,Y)\}.
\end{equation*}
$\pi $ and $\bar{\tau}$ satisfy the conditions posed by the theorem: for
all partial event $(X,Y)$ in $G_{A}$ and all formula $\alpha $ in 
$\bar{\tau}(X,Y)$, 
\begin{equation*}
\pi (\alpha )=\overline{v} (|\alpha |)=\mu (\eta (|\alpha |))=\mu (X,Y).
\end{equation*}
\end{proof}

\end{document}